\documentclass{amsart}
\usepackage{amsmath}
\usepackage{framed,comment,enumerate}
\usepackage[pdftex]{graphicx}
\usepackage{epstopdf}
\usepackage{xcolor, framed}
\usepackage{color}

\newtheorem{prop}{Proposition}[section]
\newtheorem{thm}{Theorem}[section]

\newtheorem{cor}{Corollary}[section]
\newtheorem{lem}{Lemma}[section]
\theoremstyle{definition}
\newtheorem{defi}{Definition}[section]
\newtheorem{rem}{Remark}[section]
\numberwithin{equation}{section}

\def\R {\mathbb{R}}
\def\eps{\varepsilon}

\title[The $N$-membrane problem]{Free boundary regularity in the multiple membrane problem in the plane} 

\author{Ovidiu Savin}
\address{Department of Mathematics,	Columbia University, New York, USA}
\email{savin@math.columbia.edu}

\author{Hui Yu}
\address{Department of Mathematics,	National University of Singapore, Singapore}
\email{ huiyu@nus.edu.sg}
\thanks{O.~S.~is supported by  NSF grant DMS-1500438.}

\begin{document}

\begin{abstract}
We study the regularity of free boundaries in the multiple elastic membrane problem in the plane. We prove the uniqueness of blow-ups, and that the free boundaries are $C^{1,\log}$-curves near a regular intersection point.   
\end{abstract}

\maketitle

%%%%%%%%%%%%%%%%%%%%%%%%%%%%%%%%%%%%%%%%%%%%%%%%%%%%%%%%%%%%%%%%%%%%%%%%%%%%%%%%%%%%%%%%%%%%%%%%%%%%%%

\section{Introduction}

 Given a positive integer $N$,  the $N$-membrane problem describes the shapes of $N$ elastic membranes under external forces. The membranes cannot penetrate each other, but they can coincide in a priori unknown regions, giving rise to $(N-1)$ free boundaries. The $N$-membrane problem can be viewed as a coupled system of $(N-1)$ obstacle problems with interacting free boundaries. It is the natural extension of the obstacle problem (which corresponds to the case $N=2$) to the vector valued case, and can be referred to as {\it the vectorial obstacle problem}.
 
 Mathematically, given a domain $\Omega\subset\R^d$, positive constants $\{\omega_k\}_{k=1,2,\dots,N}$, and  bounded functions $\{f_k\}_{k=1,2,\dots,N}$, we study the minimizer of the following convex functional 
 \begin{equation}\label{FirstEquation}
 (u_1,u_2,\dots,u_N)\mapsto\int_{\Omega}\sum \, \omega_k(\frac{1}{2}|\nabla u_k|^2+f_ku_k) \, dx \end{equation}
 over the class of  functions with prescribed data on $\partial\Omega$, and subject to the constraint \begin{equation}\label{SecondEquation}
 u_1\ge u_2\ge\dots\ge u_N \text{ in $\Omega$.}
 \end{equation} The function $f_k$ represents the force acting on the $k$th membrane, whose height is described by the unknown $u_k.$ Each $\omega_k $ represents the weight of the $k$th membrane. 
 
 Since the membranes cannot penetrate each other, the functions $\{u_k\}$ are well-ordered inside the domain. This leads to  the constraint \eqref{SecondEquation}. On the other hand, consecutive membranes can come in contact with each other. Between the \textit{contact region} $\{u_k=u_{k+1}\}$ and the \textit{non-contact region }$\{u_k>u_{k+1}\}$, we have the $k$th \textit{free boundary} 
$$ 
\Gamma_k:= \partial\{u_k>u_{k+1}\}.
$$
We consider the case of \textit{constant} force terms that  satisfy a \textit{non-degeneracy condition} specific in obstacle-type problems 
$$f_1>f_2>\dots>f_N.$$

The Euler-Lagrange equation is given in the form of the variational inequality 
\begin{equation}\label{EL1}
 \omega_i (v_i -u_i) \triangle u_i \le \omega_i (v_i -u_i)f_i,
 \end{equation}
which holds for all $v\in H^1(\Omega)$ that satisfy the constraint \eqref{SecondEquation}. Since the convex set defined by \eqref{SecondEquation} is invariant under addition of the same function and multiplication by the same positive number, we have further
\begin{equation}\label{avei}
\sum \, \omega_i \, \triangle u_i = \sum \omega_i f_i, \quad \quad \sum \omega_i u_i \triangle u_i = \sum \omega_i u_i f_i.
\end{equation}

Existence and uniqueness of the minimizer  were established  by Chipot and Vergara-Caffarelli \cite{CV}. They also proved that solutions are $C_{loc}^{1,\alpha}(\Omega)$ for all $\alpha\in(0,1)$. We obtained  the optimal $C^{1,1}$-regularity of solutions  and then performed a blow-up analysis in Savin-Yu  \cite{SY1}. 

The case when $N=2$ corresponds to the classical obstacle problem. Concerning this problem, there is  a large literature, see, for instance, \cite{C1,C2,M,W,CSV1,FSe}. For the case when $N=3$, the free boundary regularity was investigated recently in \cite{SY2}. The non-trivial analysis occurs near the points where the two free boundaries intersect. Exploiting a maximum principle satisfied by the pair $(u_1, -u_3)$ which is specific to $N=3$ membranes, we obtained  the sharp logarithmic rate of blow-up.  With this, we established the $C^{1,\log}$-regularity of the free boundaries near {\it regular intersections}, and the uniqueness of certain types of blow-up profiles.

In this work, we extend these results in the physical dimension $d=2$ to an arbitrary number of membranes $N$, and to all possible blow-up profiles. For arbitrary $N$, the setting is much more complicated as the complexity of the problem grows exponentially with $N$. Nevertheless, we are able to prove uniqueness of blow-ups as well as sharp free boundary regularity near a regular intersection point.
 A consequence of our results is that the free boundaries intersect tangentially if the corresponding coincidence sets have positive densities at the intersection point. This is one of the interesting features of the problem: the $(N-1)$ degrees of freedom of the problem do not usually match the degrees of freedom of the free boundaries when they intersect!

Uniqueness of blow-ups is a central problem in the regularity theory, and it is usually achieved through a differential inequality known as the log-epiperimetric inequality of the type
$$ \frac{d}{dr} W(u,r) \le - c \, W(u,r)^{\gamma} , \quad \quad \gamma <2.$$
Here $W$ represents the functional that appears in the (Weiss) monotonicity formula, translated so that it tends to $0$ as $r \to 0$.
For cones with smooth cross sections and when $W$ has analytic structure, a general method to establish the log-epiperimetric inequality  is based on the Lojasiewicz-Simon inequality. The method was developed by L. Simon \cite{S} in the setting of minimal surfaces. However, this strategy does not seem to apply in obstacle type problems as the constraint \eqref{SecondEquation} is polyhedral. The log-epiperimetric inequality in the standard obstacle problem was established by Colombo-Spolaor-Velichkov \cite{CSV1} by making use of the Fourier decomposition of the traces of $u$ on $\partial B_r$. The same authors extended their results to cones of even frequency for the thin obstacle problem \cite{CSV2}.

Recently in \cite{SY3,SY4}, we proposed an ad-hoc strategy to establish the uniqueness of certain blow-up cones in obstacle-type problems, which is inspired by our work for $N=3$. This is based on introducing approximate solutions, modeled by solutions of the linearized problem. These approximate solutions are so that they minimize the error of the right hand side in the Euler-Lagrange equation, and are used to approximate the dyadic rescalings of the actual solution $u$. Their construction usually involves solving appropriate obstacle problems on $\partial B_1$. The fact that the error cannot be improved reduces to a non-orthogonality condition, which often is given in the form of a nontrivial algebraic statement.  
The strategy is the following. 

Assume the solution $u$  is within an $\eps$ error of an approximate solution $v$ in $B_1$. Then we need to show that in a smaller ball $B_\rho$, either $u$ has a $\eps/2$ rescaled error with respect to another approximate solution $w$ (which would give a geometric convergence rate for the rescalings of $u$), or the energy of $u$ in $B_\rho$ decayed at least an $\eps^2$ amount i.e.
$$W(u,\rho) \le W(u,1) - c \eps^2.$$
This dichotomy is a consequence of the fact that $v$ is  ``the least error" approximation among functions which project in the same point on the tangent space given by the linearized equation. 
Then we establish an inequality of the type $W(u,1) \le \eps^{1+\mu}$ for some $\mu>0$, which together with the inequality above gives a discrete version of the log-epiperimetric inequality and leads to the uniqueness of blow-up limits.  

In the present work, we follow the same strategy. An important point is that in dimension $d=2$ all cones are classified, and this plays a key role in the algebra involved, see Section 4.  The construction of approximate solutions relies on the solvability of the global problem in 1D, which we investigate in Section 3. Throughout the paper we use the bold face letter notation for vectors, say $${\bf u}=(u_1,..,u_N).$$

Before we introduce our results a few simplifications are in order. We may assume that all $N$ free boundaries pass through the origin, 
$$0 \in \cap \Gamma_i,$$
since an intersection point involving fewer free boundaries can be reduced locally to the same problem with fewer membranes. Also, after subtracting the average from all $u_k$, we may assume that the average of the $u$'s and $f$'s is $0$ (see \eqref{avei}):
$$\sum \omega_k \, u_k =0, \quad \sum\omega_k \, f_k=0.$$
In \cite{SY1}, we showed that the quadratic rescalings
$${\bf u}_r(x):=r^{-2} {\bf u}(rx),$$
converge on subsequences as $r \to 0$ to a $2$-homogenous solution ${\bf p}$, i.e. a \textit{cone}. Moreover, in dimension $d=2$, we classified the family $\mathcal C_2$ of cones as extensions of 1D cones to two dimensions (see next section for more details). 

We state the main results.

\begin{thm}\label{TIntro1}
Assume that $d=2$ and ${\bf p} \in \mathcal C_2$ is a blow-up limit for $\bf u$ at the origin. Then, $\bf p$ is unique and
$${\bf u}(x)={\bf p}(x) + O(|x|^2(-\log |x|)^{-1}).$$
\end{thm}

Among the two-dimensional cones, the one of least energy is given by rotations of
 $$
 {\bf p_0}(x_2):= \frac 12 (x_2^+)^2 {\bf f},
 $$
 which represents the situation when all coincidence sets are given by the same half-plane.
 If ${\bf p}_0$ appears as a blow-up limit at the origin then we say that $0$ is a \textit{regular intersection point} for the free boundaries $\Gamma_i$. Near these points, the free boundaries enjoy the following regularity:

 \begin{thm}\label{TIntro2}
 Assume $d=2$ and $$|{\bf u}-{\bf p}_0| \le \eps_0 \quad \mbox{in $B_1$}$$ for a constant $\eps_0$ depends on $N, {\bf f}$ and  $\omega$.

 Then each $\Gamma_i$ is a $C^{1,\log}$-curve in $B_{1/2}$.  \end{thm}

The paper is structured as follows. In Section 2, we introduce the notations, and collect some general facts about the maximum principle and the optimal regularity of solutions. In Section 3, we study the global 1D problem which is crucial to our analysis. In Sections 4 and 5, we prove Theorem \ref{TIntro1} for those non-degenerate cones (connected cones) ${\bf p}$ for which all their coincidence sets have non-empty interiors. In Section 6 we prove Theorem \ref{TIntro1} for all other degenerate cones. Finally, in Section 7 we prove Theorem \ref{TIntro2}.

We conclude the introduction with a game theoretical interpretation of the $N$-membrane problem. Suppose there are $N$ players $P_1$,..,$P_N$ which hold $N$ tickets $1,2,..,N$ and a token that moves on a lattice in $\Omega$. Each round the token moves randomly to an neighboring vertex and the players can interchange their tickets according to the following rule: the player with the ticket 1 can choose any ticket he wishes, after that the player with the ticket 2 can choose from the remaining $N-1$ tickets and so on. Moreover, in order for a player to hold onto the ticket 1 for one round he needs to pay the amount $f_1$, and for the ticket 2 the amount $f_2$, etc. When the token exits the domain, the payoff for the ticket $k$ holder  is given by the boundary data $\varphi_k$. If all players optimize their strategies then the solution $u_k$ to the discrete $N$-membrane problem (with weights $\omega_k=1$) represents the expected payoff of the player holding the ticket $k$, while the coincidence sets give the optimal strategies on the exchange of tickets. 

%%%%%%%%%%%%%%%%%%%%%%%%%%%%%%%%%%%%%%%%%%%%
%%%%%%%%%%%%%%%%%%%%%%%%%%%%%%%%%%%%%%%%%%%%%

\section{Notation and preliminaries}
In this section we introduce the notations used through the paper, and collect some basic properties of solutions to the $N$-membrane problem, such as optimal regularity, maximum principle and introduce the cones in one and two dimensions. 

{\bf Notation.} 

 \noindent ${\bf u}=(u_1,..,u_N)$. 

\noindent ${\bf 1}=(1,1,..,1)$.
 \begin{equation}\label{ugev}\mbox{${\bf u} \ge {\bf v}$ means $u_i \ge v_i$ for all $i$.}
\end{equation}
\noindent  For $I \subset \{1,..,N\}$,  $u_I$ denotes the average of $u_i$ with $i \in I$
\begin{equation}\label{u_ave}
u_I:= \sum_{i \in I} \frac{\omega_i}{\sum_{I} \omega_j} \, u_i.
\end{equation}
\noindent  $\mathcal P$ denotes the collection of 1D cones, see Definition \ref{calP}.

\noindent $\mathcal P^c \subset \mathcal P$ are the connected 1D cones, see Definition \ref{calP}.

\noindent $B({\bf p})$ is the space associated to the branches ${\bf p} \in \mathcal P^c$, see Definition \ref{B(p)}.

\noindent ${\bf h}(x,{\bf b})$ is the global 1D solution with linear asymptotics given by ${\bf b} \in B({\bf p})$, see Definition \ref{hxb3}.

\noindent $\tau \in B({\bf p})$ is generated by the 1-translation, see Definition \ref{tau}.

\noindent ${\bf e}({\bf b})$ is the error function, see Definition \ref{ER}.

\noindent ${\bf p}(x,{\bf b})$ the approximate solution generated by ${\bf b}$, see Definition \ref{pb}.

\noindent ${\bf p}(x,{\bf b}_0, {\bf b}_1)$, see Definition \ref{pbb}.

\noindent $\mathcal S (r,{\bf p}, \eps )$, see Definition \ref{S}.

\noindent $W({\bf u},r)$ the Weiss functional, see Section 5.

\noindent ${\bf p}_*\in \mathcal P \setminus \mathcal P^c$ denotes a degenerate 1D cone.

\noindent ${\bf p}_*^i \in \mathcal P^c$ are the connected 1D cones which make ${\bf p}_*$, see Section 6.

\noindent $\mathcal S (r,{\bf p}_*, \eps )$, see Definition \ref{S_sec}.

\noindent $\sigma$-connected, see Definition \ref{scon}

\noindent We denote by $c_i$, $C_i$ constants depending on $N$, $d$, ${\bf f}$, $\omega$, and call them universal constants.

\

If $h$ is a function with $\triangle h= const.$, then ${\bf u} - h {\bf 1}$ solves the $N$-membrane problem with forces ${\bf f} - (\triangle h){\bf 1}$, see
\eqref{EL1}-\eqref{avei}. Thus, without loss of generality we assume throughout that the $f$'s have average $0$
$$ \sum \omega_i f_i =0,$$
and by \eqref{avei}, $\sum \omega_i u_i$ is harmonic. 

Often we subtract the average of the $u_i$ from each function so that we reduce to the case $\sum \omega_i u_i=0$. When this holds we say that $u$ solves the Problem $P_0$.

\begin{defi}\label{P0}
We say that ${\bf u}$ solves the problem $P_0$ if it is a solution to the $N$-membrane problem and also $\sum \omega_i u_i=0$.
\end{defi}
The Euler-Lagrange equation gives that in an open region where $l$ membranes coincide $u_m < u_{m+1}= u_{m+2}=..=u_{m+l} < u_{m+l+1}$, the common function $u_{m+1}$ satisfies
$$\triangle u_{m+1} = f_I, \quad I:=\{m+1,..,m+l \},$$
i.e. the force acting on each of the $l$ membranes in this coincidence region is the average of the $l$ forces $f_i$. 

\

{\bf Optimal regularity.} Existence and uniqueness of solutions in $H^1(\Omega)$ follows easily from the standard methods in the calculus of variations. The optimal $C^{1,1}$ regularity of solutions was obtained in \cite{SY1}. We sketch the proof for completeness. 
We show that $u_i \in C^{1,1}_{loc}$ and
\begin{equation}\label{EL2}
\triangle u_i = \sum_{j \le i \le k} \, f_{A_{jk}} \, \chi_{A_{jk}}, \quad \quad A_{jk}:=\{u_{j-1}<u_j=..=u_k< u_{k+1} \}.
\end{equation}

{\it Lipschitz regularity.}
If ${\bf v} \in H^1(B_1)$ in another solution, then by adding the variational inequalities \eqref{EL1} for ${\bf u}$ and ${\bf v}$ we find
$$ \omega_i (v_i- u_i) \triangle (v_i - u_i) \ge 0 \quad \Longrightarrow \quad   \triangle \left(\omega_i (v_i-u_i)^2 \right) \ge 0,$$
hence $ \sum \omega_i (v_i-u_i)^2$ is subharmonic. This shows that
$$\|{\bf v}-{\bf u}\|_{L^\infty(B_{1/2})} \le C \|{\bf v}-{\bf u}\|_{L^2(B_1)}.$$
Taking $\bf v$ to be a translation of ${\bf u}$, we obtain
$$\|\nabla {\bf u}\|_{L^\infty(B_{1/2})} \le C \|\nabla {\bf u}\|_{L^2(B_{1})}.$$

{\it $C^{1,1}$ regularity.}

\begin{lem}\label{c11} Assume ${\bf u}$ solves the $N$-membrane problem in $B_1$. Then
\begin{equation}\label{fbd}
|\triangle u_m| \le C |{\bf f}|,
\end{equation}
\begin{equation}\label{fbd2}
\|{\bf u}\|_{C^{1,1}(B_{1/2})} \le C \left( \|{\bf u}\|_{L^\infty(B_1)} + |{\bf f}| \right).
\end{equation}
\end{lem}
\begin{proof} We use induction induction on $N$. 
The case $N=1$ is trivial. 

For $N>1$, after subtracting the average, we may assume that $\sum \omega_i u_i=0$, and say also that $|\bf f|=1$.
We start with \eqref{fbd}.

The set where all membranes coincide is 
$$K:=\{u_i=0, \quad \forall i\}=\{ u_1=u_N\}.$$
The inequality \eqref{EL1} implies $\triangle u_1 \le f_1$, $\triangle u_N \ge f_N$ hence $\triangle (u_1-u_N) \le f_1- f_N$. 
This means that $w:=u_1-u_N \ge 0$, satisfies $ \triangle w \le C $ in $\Omega$ and, by the induction hypothesis $|\triangle w| \le C $ in the set 
$\{w>0\}=\Omega \setminus K$. This shows that $w$ solves a scalar obstacle problem with right hand side bounded in $L^\infty$, which implies the standard quadratic growth away from its zero set
$$w(x) \le C \, d(x,K)^2,$$
where $d(x, K)$ denotes the distance from $x$ to the set $K$. Then $|u_1|,|u_N| \le w$ satisfy the same inequality, and it holds for all other $|u_m|$. This shows $|\triangle u_m| \le C $ on $K$ in the viscosity sense, while outside $K$ the inequality holds (in the viscosity sense) by the induction hypothesis. In conclusion \eqref{fbd} is proved.

As a consequence $u_m-u_{m+1} \ge 0$ solves an obstacle problem with a $L^\infty$ right hand side, and it satisfies the standard quadratic growth behavior
\begin{equation}\label{quadsep}
u_{m} - u_{m+1} \le C \,  d(x, \Gamma_k)^2 \quad \mbox{in $\{u_m> u_{m+1}\}$}.
\end{equation}

Next we prove \eqref{fbd2} by showing that each function $u_m$ admits a tangent paraboloid by above/below of opening $1+ \|u\|_{L^\infty}$. For simplicity we prove this at the origin. 

Let $r \ge 0$ denote the radius of the smallest ball around the origin $B_r$ which intersects all free boundaries $\Gamma_i$. Notice that in $B_r$ the problem decouples into two multi-membranes problems involving fewer membranes than $N$. 

If $r \ge 3/4$, then we can apply the induction hypothesis in $B_r$ and get the desired conclusion  in $B_{1/2}$. 
If $r \in (0,3/4)$, then by \eqref{quadsep} we conclude that $u_{m} - u_{m+1} \le C (r^2 + |x|^2)$ for all $m$. Since the average of the $u$'s is 0 we find 
$$|u_m| \le C (r^2 +|x|^2).$$ In $B_r$ we may apply again the induction hypothesis (for the rescaling ${\bf u}(rx)/r^2$). Then we conclude that $u_m$ admits a global tangent polynomial of opening $C$ by above/below at the origin (outside $B_r$ we use the inequality above).

Finally, if $r=0$, we obtain as above $|u_m| \le C |x|^2$ which gives again the desired estimate.

\end{proof}

\begin{rem}
Lemma \ref{c11} implies \eqref{EL2} by considering Lebesgue points for $A_{jk}$ where ${\bf u}$ is twice differentiable. If we assume that ${\bf f}$ satisfies the nondegenerate condition $f_1>f_2>..>f_N$ then the right hand side for $\triangle(u_{m}-u_{m+1})$ is positive, and we obtain also the quadratic growth by below 
$$\max_{B_r(x_0)} (u_{m} - u_{m+1}) \ge c r^2 \quad \mbox{if $x_0 \in \Gamma_m$,}$$
for some $c>0$ universal.
\end{rem}

{\bf Maximum Principle.} The maximum principle takes the following form in the setting of the $N$-membrane problem.

\begin{lem}[Maximum Principle] 
If $\bf u$ and $\bf v$ are 2 solutions with $\bf u \ge \bf v$ on $\partial \Omega$, then $\bf u \ge \bf v$ in $\Omega$. 

Moreover, if $u_i(x_0)=v_i(x_0)$ for some $x_0 \in \Omega$, then $u_i=v_i$.
\end{lem}

\begin{proof}
Let $I \subset \{1,..,N\}$ be the set of $m$'s for which $u_m(x_0)=u_i(x_0)$ and similarly define $J$ the set of membranes that coincide with $v$ at $x_0$. We have $\max I \ge \max J$, $\min I \ge \min J$. Then the average function $u_{I \cap J}$ (see \eqref{u_ave})
satisfies
$$ \triangle u_{I \cap J} \le f_{I \cap J} $$
in a neighborhood of $x_0$, since we may perturb the membranes $u_m$ with $m \in I \cap J$ upwards by a positive function $\eps \varphi$, $\varphi \in C_0^\infty (B_r(x_0))$ and 
keep satisfying the constraint \eqref{SecondEquation}.  Similarly,
$$ \triangle v_{I \cap J} \ge f_{I \cap J}.$$ Since $u_{I \cap J} \ge v_{I \cap J}$ and they coincide at $x_0$ we find that they coincide in $B_r(x_0)$.

\end{proof}

{\bf 1D and 2D cones.}

\begin{defi}\label{calP}
We denote the space of 1D cones by $\mathcal P$:
$$\mathcal P=\{ {\bf p} | \quad \mbox{$\bf p$ is a homogenous of degree 2 solution, and $0 = \cap \Gamma_k$}\}.$$
We denote by $\mathcal P^c$ the solutions $\bf p \in \mathcal P$ which 
are non-trivially \textit{connected }in the sense that each coincidence set $\Lambda_m:=\{u_m=u_{m+1}\}$ is a half-line (or equivalently has nonempty interior),
$$\mathcal P^c=\{ {\bf p }\in \mathcal P| \quad int \, \Lambda_m \ne \emptyset \quad \quad \forall m \le N-1\}.$$

\end{defi}

There are $3^{N-1}$ elements in $\mathcal P$, since there are 3 options for each of the coincidence sets $\Lambda_m$: $(-\infty,0]$, $\{0\}$, $[0,\infty)$, and there are $2^{N-1}$ elements in $\mathcal P^c$. 

A particular solution in $\mathcal P^c$ is $\bf p_0$ which has the components $p_i= \frac{f_i}{2} (x^+)^2$. 
It turns out that $\bf p_0$ and its reflection $\bf p_0(-x)$ are the least energy solutions among all $\bf p \in \mathcal P$.

\

In \cite{SY1} we showed that the space of 2D cones $\mathcal C_2$ is generated by 1D cones in the following way. If ${\bf p} \in \mathcal P^c$ then its 2D extension coincides with ${\bf p}(x_2)$ up to rotations. 

If ${\bf p}_*   \in \mathcal P \setminus \mathcal P^c$ (i.e. a {\it degenerate} cone) then we first decompose ${\bf p}_*$ as a union of $m \ge 2$ connected cones in $\mathcal P^c$. Each of these cones is extended to 2D, and then modified by a harmonic function and an angle of rotation, see Section 6 for more details.

\

{\bf A convergence lemma.} We state a lemma about sequences and the convergence of series, which we use in the main result. In out setting $w_n$ will represent the Weiss energy of $u$ in the ball of radius $\rho^N$, while $\eps_n$ the rescaled error between $u$ and an approximate solution.

\begin{lem}\label{Sequences}
Let $w_n$ and $\eps_n$ be two sequences of real numbers between $0$ and $1$. Suppose that 
$$w_{n+1}\le C_0 \, \eps_n^{3/2},$$ and either
$$w_{n+1}\le w_n \quad\mbox{ and} \quad \eps_{n+1}=\eps_{n}/2,$$
or $$w_{n+1}\le w_n-c\eps_n^2 \quad \mbox{ and} \quad \eps_{n+1}=C \eps_n .$$Then 
\begin{equation}\label{EvenSum}\sum_{n\ge k}\eps_n \le M k^{-1},\end{equation}
for some $M$ depending only on $c$, $C$, $C_0$.
\end{lem} 
\begin{proof}
We only sketch the proof (see \cite{SY3} for more details). 

The sequence $a_n:=w_n + c' \eps_n^2$, satisfies $a_{n+1} \le a_n - c \eps_n^2 \le a_n - C a_n^{4/3}$ which implies $a_n \le C n^{-3}.$ The conclusion follows by adding the inequalities $$\eps_n \le C (a_n -a_{n+1})^{1/2}.$$
\end{proof}

%%%%%%%%%%%%%%%%%%%%%%%%%%%%%%%%%%%%%%%%%%
%%%%%%%%%%%%%%%%%%%%%%%%%%%%%%%%%%%%%%%%%%%

\section{The 1D Problem}

In this section, we study the $N$-membrane problem in 1D. For each cone ${\bf p} \in \mathcal P^c$ and vector ${\bf b}$ associated to the branches of ${\bf p}$, we show that there is a unique global solution with linear asymptotics given by ${\bf b} x$ at $\pm \infty$. We also introduce the error function ${\bf e}({\bf b})$, which plays an important role in the study of approximate solutions.

\

In the 1D problem, each component of the solution is piecewise quadratic, and the difference between consecutive membranes is convex. This means that the coincidence set $\{u_m=u_{m+1}\}$ is an interval. 
Recall that $\mathcal P^c$ represents the connected 1D cones, see Definition \ref{calP}.
If ${\bf p} \in \mathcal P^c$, then the graphs of all the components of $\bf p$ consists of 
$(N+1)$ disjoint half quadratics starting at the origin i.e. $a (x^+)^2$ or $a (x^-)^2$. 
This is because any two consecutive graphs of the $p_i$ have precisely a half quadratic in common. 
We call these disjoint quadratics the {\it branches} of $\bf p$. 
The right branches of $p$ are the graphs over $[0,\infty)$ and the left branches the ones over $(- \infty,0]$. 
The condition $\sum \omega_i p_i=0$ implies that the right (respectively left) branches average to 0 when counting their weights and multiplicities.  

We associate a real number $b_k$ to each of the branches of $\bf p$ with the compatibility condition that the average of these numbers on the right (respectively left) branches equals 0. The collection of these $b_k$ is denoted by ${\bf b} \in B({\bf p})$.

\begin{defi}\label{B(p)}
For each ${\bf p}\in\mathcal{P}^c$, 
the space $B({\bf p})$ consists of vectors $${\bf b}=(b_1^-, b_2^-,..,b_N^-, b_1^+,..,b_N^+),$$ with the property that $\sum \omega_i b_i^-=\sum \omega_i b_i^+=0$ and 
$$ \mbox{$b_i^-=b_{i+1}^-$ if $p_i=p_{i+1}$ on $(-\infty ,0]$, \quad \quad  $b_i^+=b_{i+1}^+$ if $p_i=p_{i+1}$ on $[0, \infty)$.}$$
\end{defi}

Clearly, $B({\bf p}) \subset \R^{2N}$ is a $N-1$ dimensional linear subspace. 

We want to solve the $N$-membrane problem after perturbing the branches of a solution $\bf p \in \mathcal P^c$ by $x {\bf b}$.

\begin{prop}\label{1D}
Given ${\bf b} \in B({\bf p})$, there exists a unique solution $u$ to the problem $P_0$ in $\R$ which satisfies
$$u_i = p_i + b_i^{\pm} x^{\pm} + o(|x|)  \quad \mbox{as $x \to \pm \infty$}$$
where $b_i=b_i^\pm$ is the number associated to the branch of $p_i$.
\end{prop}

\begin{proof}
We first show the existence. 

We solve the problem in the interval $[-R,R]$ with boundary data $u_i=p_i + b_i x$ and obtain a solution ${\bf u}^R$, and 
then let $R \to \infty$. We need some uniform estimates. 

Let $M > \max |b_i|$, and let $t_0$ be the first value as we decrease $t$ for which the inequality
$${\bf p} + (t+M |x|) {\bf 1} > {\bf u} ^R \text{ on $[-R,+R]$}$$
fails. When $t=t_0$ then we need to replace $>$ with $\ge$ above and equality holds at some $x_0$ for some $i$-component. 

Notice that $t_0 \ge 0$ which follows from the inequality written at $x=0$ and $\sum \omega_i u_i=\sum \omega_i p_i=0$. 
The left hand side is a solution to our problem in each interval $(-R,0)$, $(0,R)$ 
and by the strong maximum principle it follows that the first contact point must be $x_0=0$, 
since at the end points $\pm R$ we have strict inequality by the choice of $M$.  

We claim that $t_0 \le C \,  M^2$ with $C$ a universal constant. We choose $K= \delta^{-1} M$ with $\delta>0$ the universal constant from Lemma \ref{MP1d} below, and then define $\bf v$ as the translation of ${\bf u}$
$${\bf v}:={\bf u}^R- (t_0 + M K) {\bf 1}.$$
We have
$$ {\bf p} \ge {\bf v} \quad \mbox{in $[-K,K]$,}$$
and  
$$\quad v_i(0) = - MK = p_i(0) - \delta K^2.$$
By Lemma \ref{MP1d} (rescaled) we find ${\bf v}(0) \ge - K^2 {\bf 1}$ which means  $u_j^R(0) \ge  t_0 + MK - K^2$ 
and the claim follows from $\sum \omega_i u^R_i(0)=0$. A symmetric argument gives 
$$
{\bf p} + (CM^2+M |x|) {\bf 1} \ge {\bf u} ^R \ge  {\bf p} - (CM^2+M |x|) {\bf 1}.
$$
Since $u^R_{i+1}-u^R_i$ has to grow quadratically away from the free boundary, it follows that if $p_i=p_{i+1}$ say on $[0, \infty)$ then 
$u^R_i=u^R_{i+1}$ on $[CM,R)$ for some $C$ universal. In particular $u^R_i$ and $p_i + b_i x$ have the same constant as second derivative on $
[CM, R)$. Their difference is at most $C M^2$ as at the end points of the interval. 
As $R \to \infty$ we can extract a subsequence which converges uniformly on each compact set and has the asymptotic expansion required. 

For the uniqueness, we argue as above and obtain that $u_i$ has the same second derivative as $p_i + b_i x$ in a neighborhood of $\infty$ (or $-\infty$) and therefore they must differ by a constant. Thus if $\bf v$ is another solution, $\sum \omega_i (u_i-v_i)^2$ is convex and bounded and therefore it is a constant. In particular $\nabla (u_i - v_i)=0$ for each $i$, thus $u_i-v_i$ is constant for each $i$. Since the branches of $\bf u$ and $\bf v$ are connected we find that these constants are independent of $i$, and since their average is $0$, they all must be $0$.

\end{proof}

We give a quantified version of the strong maximum principle for solutions near $\bf p \in \mathcal P^c$.
\begin{lem}\label{MP1d}
Let $\bf p \in \mathcal P^c$ and let ${\bf v}$ be a solution of our problem (not necessarily of average $0$) with ${\bf p} \ge {\bf v}$ in $[-1,1]$, $v_i(0) \ge - \delta$ for some $i$. Then $v_j(0) \ge -1$ for all $j$, provided that $\delta$ is sufficiently small.
\end{lem}

\begin{proof}
The inequality is clear if $j \le i$. It suffices to show that the collection of the graphs of the $v_j$ with $j \ge i$ are all connected in the strip $\{|x| \le c\}$ for some $c$ small. Assume not, and them let $l \ge i$ be the last membrane connected to $v_i$ in $[-c,c]$. Then $v_1,.,v_l$ are uniformly bounded in $[-c,c]$, and solve the $l$-membrane problem in $[-c,c]$. By compactness (for fixed $l$), as $\delta \to 0$ we obtain a limiting solution $\tilde v$ of the $l$-membrane problem which is below $(p_1,..,p_l)$ and with $\tilde v_i(0)=p_i(0)=0$. Since $l < N$, $(p_1,..,p_l)$ is a strict supersolution to the $l$-membrane problem, and we contradict the maximum principle between $p$ and $\tilde v$.

\end{proof}

\begin{defi}\label{hxb3}
Given ${\bf p} \in \mathcal P^c$ and ${\bf b} \in B({\bf p})$, we denote by 
$${\bf h} (x,{\bf b})$$
the unique solution ${\bf u}$ from Proposition \ref{1D} to the problem $P_0$ which has linear coefficients $\bf b$ in its asymptotic expansion at $\pm \infty$
$$u_i = p_i + b_i x + o(|x|)  \quad \mbox{as $x \to \pm \infty$}.$$
\end{defi}

\begin{defi}\label{tau}
Notice that ${\bf p}(x+ 1)$ has linear coefficients $\tau_i := p_i'/x$ in its expansion at $\pm \infty$. Hence if ${\bf b}=s {\bf{\tau}}$ then
$$ {\bf h}(x, s {\bf \tau})={\bf p}(x+s),$$
or more generally
$$ {\bf h}(x, {\bf b}+s {\bf \tau})={\bf h}(x+s,{\bf b}).$$

\end{defi}

\begin{lem}\label{L2.2}
The function ${\bf h}(x, {\bf b})$ is homogenous of degree 2 in the variables $x$ and ${\bf b}$, and is $C^{1,1}$ and piecewise quadratic in the $x$ variable. 

Moreover,
$$h_i = p_i + b_i x + O(\|{\bf b}\|^2),$$
and outside the interval $[-C\|{\bf b}\|, C\|{\bf b}\|]$ we have
 $$h_i = p_i + b_i x + e_i,$$
 with $e_i$ a constant which depends only on the branch.  
\end{lem}
\begin{defi}\label{ER}
We refer to the function ${\bf b} \mapsto {\bf e}$ which maps $B({\bf p})$ to $B({\bf p})$ as {\it the error function} (which is a homogenous of degree 2 map).
\end{defi}
It turns out that ${\bf h}(x,{\bf b})$ is $C^{1,1}$ in the $\bf b$ variable as well. The proof of this fact is technical and can be skipped on a first reading.
 
\begin{lem}\label{L2.3}
The function ${\bf h}(x, {\bf b})$ is piecewise quadratic and of class $C^{1,1}$ in both variables $x$ and ${\bf b}$. In particular the error map ${\bf e}({\bf b})$ is piecewise quadratic in $\bf b$.
\end{lem}
\begin{proof}
Each solution $\bf u$ to the problem $P_0$ which is asymptotic to $\bf p$ at infinity, in the sense that $R^{-2}{\bf u}(Rx) \to {\bf p}$ must be of the form ${\bf h}(x, {\bf b})$ and is uniquely determined by $\bf b$. 

On the other hand each such solution is also uniquely determined by the location of the free boundaries $\Gamma_i$. For example $u_1$ and $u_2$ 
coincide on the side of $\Gamma_1$ where their corresponding branches agree and they must differ on the other side of $\Gamma_1$. So if we know the locations of all the $
\Gamma_i$, $1 \le i \le N-1$, then we know in each of the corresponding subintervals determined by the $\Gamma_i$ which membranes 
coincide, and thus the second derivatives of all the $u_i$ are uniquely determined. In other words if we arrange the free boundary points in increasing order $\Gamma_{i_1} \le \Gamma_{i_2} \le \Gamma_{i_{N-1}}$, then each $u_k''$ is determined on the interval $[\Gamma_{i_j}, \Gamma_{i_{j+1}}]$ by the permutation $\pi=\{i_1,..,i_{N-1}\}$ of $\{1,2,..,N-1\}$.
We can then integrate these second derivatives and construct a solution $\bf u$ to the problem $P$ with free boundaries $\Gamma_i$. Since the graphs of all the membranes are connected the solution $\bf u$ is unique up to a linear function. We explain more in detail how to construct ${\bf u}$ inductively in the following way. 

Assume that the top membrane $p_1$ of $\bf p$ is free on the left and has the common branch with $p_2$ on the right. Then we construct $u_1$ on the left of $\Gamma_1$ as $\frac {f_1}{2} (x-\Gamma_1)^2$ and then on the right of $\Gamma_1$ we need to add to this quadratic a linear combination of terms $[(x-\Gamma_k)^+]^2$ according to values of $u_1''$ on the subintervals $[\Gamma_{i_j}, \Gamma_{i_{j+1}}]$ to the right of $\Gamma_1$. Then we construct $u_2$ as equal to $u_1$ on the right side of $\Gamma_1$ and on then on the left of $\Gamma_1$ we need to adjust it by adding to $u_1$ a linear combination of terms $[(x-\Gamma_k)^-]^2$ according to the values of $u_2''$ on the subintervals to the left of $\Gamma_1$. 
Then we define $u_3$ as equal to $u_2$ on the side of $\Gamma_2$ where the branches of $p_2$ and $p_3$ coincide, and modify it on the other side 
of $\Gamma_2$ according to the values of $u_3''$. We continue this process till $u_N$. By construction $u_1 \ge u_2 \ge .. \ge u_N$ ( since $u_k'' \ge 
u_{k+1}''$ which is a consequence of nondegeneracy), and the Euler-Lagrange equations are satisfied, hence $\bf u$ is a solution of the problem $P$ 
with the given free boundaries $\Gamma_k$.   
By construction each $u_i$ is of the form
\begin{equation}\label{ui}
u_i = \sum_k \mu^+_{ki} [(x-\Gamma_k)^+]^2 + \mu^-_{ki} [(x-\Gamma_k)^-]^2
\end{equation}
where the coefficients $\mu_{ki}^\pm$ are determined only by the permutation $\pi$. 
We obtain a solution to $P_0$ after subtracting their total average from each one of them. The corrected $u_i$ have the same form as above.
The corresponding vector $\bf b$ for this solution is obtained from the asymptotic expansion of the $u_i$'s at $\pm \infty$, which means that ${\bf b}$ is a linear combination of the $\Gamma_i$ with coefficients depending on the $\mu_{ik}^\pm$. Since ${\bf b}$ is uniquely determined by the $\Gamma_i$'s it follows that the map $(\Gamma_1,..,\Gamma_{N-1}) \mapsto {\bf b}$ is an invertible linear map on each open region of $\R^{N-1}$ where the $\Gamma_i$ do not change the order. This linear map depends only on the permutation $\pi$ and in each such region $\Gamma_k$ is a linear function of ${\bf b}$.

We view the function constructed above as a function of $N$ variables ${\bf u}(x, {\bf \Gamma})={\bf p}(x,{\bf b})$, and notice that ${\bf u}(x,{\bf \Gamma})$ is purely quadratic in its variable in each of the $N!$ convex polyhedral regions determined by the relative orders between the variables $x$, $\Gamma_1,..,\Gamma_{N-1}$. In each such region ${\bf \Gamma}= A_{\pi} {\bf b}$ for an invertible linear map $A_\pi$. Thus, when viewed as a function of $(x, {\bf b})$, ${\bf u}$ is still purely quadratic in its variables in the corresponding $N!$ polyhedral convex regions in the $(x, \bf b)$ variables. 

\

{\it Step 2:} ${\bf u}$ is $C^{1,1}$ in the $(x, {\bf b})$ variables.

 It suffices to show that the normal derivatives of the quadratic polynomials on each side of a common $N-1$ dimensional face between 2 adjacent regions coincide. Then $C I \ge D_{(x,{\bf b})}^2 {\bf p} \ge C I$ except on a set of dimension $N-2$, and this inequality can then be extended by continuity on the remaining lower dimensional set as well.

We consider a point $(x_0,{\bf b}_0)$ on a common $N-1$ dimensional face between two regions. Let ${\bf u}_0(x)={\bf p}(x,{\bf b}_0)$ be the corresponding solution for 
${\bf b}_0$ and let  ${\bf \Gamma}_0$ be the free boundary vector associated with ${\bf u}_0$. In the $(x,{\bf \Gamma})$ variables, a common $N-1$ dimensional face between two regions corresponds to the case when two of the $N$ coordinates of $(x, {\bf 
\Gamma})$ coincide and all the others are different. 

\

{\it Case 1:} $x_0$ coincides with $\Gamma_{0,k}$. 

As we let $x$ vary near $x_0$ and keep ${\bf \Gamma}_0$ fixed, the derivatives of ${\bf u_0}$ match at $\Gamma_{0,k}$ since ${\bf u}_0$ is a $C^{1,1}$ function. This means that the directional derivative with respect to the $x$-direction at $(x_0,{\bf b}_0)$ agree. This direction is transversal to the face $x= \Gamma_k$ (since $\Gamma_k$ is linear in ${\bf b}$ near $(x_0,{\bf b}_0)$) and the conclusion follows.

\

{\it Case 2:} $\Gamma_{0,k} =\Gamma_{0,l}$ for some $k < l$. 
We study the behavior of the solution ${\bf u}$ as we vary $\bf \Gamma$ in an $\eps$ neighborhood near ${\bf \Gamma}_0$.  

If $u_{0,k}(\Gamma_{0,k}) > u_{0,l}(\Gamma_{0,l})$ then there is no change in the topology of the graph of ${\bf u}$ as we vary ${\bf \Gamma}$. This means that the right hand sides for ${\bf u}''$ in the subintervals determined by ${\bf \Gamma}$ are not affected when $\Gamma_k$ and $\Gamma_l$ cross each other. The coefficients $\mu_{ij}^\pm$ in \eqref{ui} remain the same on either side of $\Gamma_k=\Gamma_l$ and the two polynomials coincide.

Next we assume that  $u_{0,k}(\Gamma_{0,k}) = u_{0,l}(\Gamma_{0,l})$, and denote by $Z \in \R^2$ the point on the graph of ${\bf u}_0$ where $k$th and $l$th membranes coincide. We prove our claim by extending the solution given by \eqref{ui} when $\Gamma_k \le \Gamma_l$ to a whole $\eps$ - neighborhood of ${\bf b}_0$ and then show that it differs from the exact solution by at most $C \eps^2$. 

Let $v(x, {\bf b})$ denote the right hand side of \eqref{ui} corresponding to the permutation $\pi$ with $\Gamma_k< \Gamma_l$, where $\Gamma_k$ are viewed as linear functions of ${\bf b}$. 
When $\Gamma_k({\bf b}) \le  \Gamma_l({\bf b})$ then $v$ is the solution to the problem $P_0$ (with asymptote $\bf b$). However, when $\Gamma_k >  \Gamma_l$ then ${\bf v}$ might fail to solve our problem near $Z$. We collect hear the properties of $\bf v$ in this case:

1)  By construction $\bf v$ is a $C^{1,1}$ function and ${\bf v}''$ is constant in each of the $N$ subintervals defined by ${\bf \Gamma}$.

2) $|{\bf\Gamma}-{\bf \Gamma_0}| = O(\eps)$ and $|{\bf v} - {\bf u}_0 |= O(\eps)$ on any compact interval.

3) The quadratic polynomial expressions in $(x,{\bf b})$ that define ${\bf v}$ in the open subintervals of $\Gamma$ remain constant as we exchange the order of $\Gamma_k$ and $\Gamma_l$, except for the ones in the interval between $\Gamma_l$ and $\Gamma_k$. 
Outside this interval the membranes of $\bf v$ that coincide when $\Gamma_k \le  \Gamma_l$ continue to coincide, and their right hand sides remain constant. In particular, ${\bf v}$ has the vector ${\bf b}$ in its asymptotic expansion at $\pm \infty$, and its average is $0$ away from  $[\Gamma_l,\Gamma_k]$. 

4) In a neighborhood of the interval $[\Gamma_l,\Gamma_k]$, for the membranes $v_i$ for those $i$'s for which $Z$ does not belong to the graph of 
the $i$th membrane of ${\bf u}_0$, their polynomial expressions remain constant. Indeed, for such $i$, $u_i''$ has no discontinuity at $\Gamma_k$ or $
\Gamma_l$ thus $\mu_{ik}^+=\mu_{ik}^-$ and $\mu_{il}^+=\mu_{il}^-$, and the orders of the $\Gamma_k$, $\Gamma_l$ do not affect the polynomial 
expressions for $v_i$.

5) Let $J$ denote the indexes of the membranes of ${\bf u}_0$ which pass through $Z$. If $j \in J$ and ${\bf u}$ is a solution near ${\bf u}_0$ then near 
$\Gamma_{0,k}$ we have 
$$u_j = u_k \quad \mbox{ if $j \le k$}, \quad u_j=u_{k+1} \quad \mbox{ if $k+1 \le j \le l$}, \quad u_j =u_{l+1} \quad \mbox{if $j \ge l+1$}.$$
The same equalities hold if we replace ${\bf u}$ by ${\bf v}$. Indeed, by 3) the equalities hold in this neighborhood outside the interval $[\Gamma_l,\Gamma_k]$. They 
hold also inside this interval which is a consequence of the fact that the difference between two $v_j$'s is a $C^{1,1}$ function with constant second 
derivative.  

Thus there are 3 different profiles for the functions $v_j$ with $j \in J$ which do not satisfy 
the correct Euler-Lagrange in $[\Gamma_l,\Gamma_k]$. These 3 profiles are connected either at $\Gamma_k$ or $\Gamma_l$, since by 3) 
$v_k=v_{k+1}$ and $v_{l+1}=v_l$ either to the left of $\Gamma_l$ or the right of $\Gamma_k$. The 3 profiles are 
uniformly $C^{1,1}$ thus they differ by at most $C \eps^2$ in this interval.

We remark that the $v_j$ with $j \in J$ (the 3 profiles) might not be monotone with respect to $j$. However, outside a $C \eps$ neighborhood of $[\Gamma_l,
\Gamma_k]$ they become ordered with respect to $j$ due to the nondegeneracy condition that holds outside this interval.

\

Now we prove that $|{\bf v} - {\bf p}(x,{\bf b})| \le C \eps^2$. Let $t_0$ be the first value as we decrease $t$ for which the inequality
$ {\bf p} + t > {\bf v},$
fails. Since ${\bf p}$ and ${\bf v}$ have the same asymptotic expansion at $\pm \infty$, and ${\bf v}$ is a solution except on the interval 
$[\Gamma_l,\Gamma_k]$ for the $v_j$'s with $j \in J$, it follows that there exists $x_0$ in this interval for which $v_j(x_0)=p_j(x_0) + t_0$. Since all 
$v_j$ and all $p_j$ are connected in this interval and are uniformly $C^{1,1}$ if follows that $|v_j - (p_j + t_0)| \le C \eps^2$ for all $j \in J$ in an $\eps$-
neighborhood of the interval $[\Gamma_l,\Gamma_k]$. However, outside this neighborhood both graphs of ${\bf v}$ and $\bf p$ solve the problem $P$ 
(with the same asymptotic expansion at $\pm \infty$), hence this inequality can be extended everywhere. Now $|t_0| \le C \eps^2$ is a consequence of the null average of ${\bf v}$ and $\bf p$ outside $[\Gamma_l,\Gamma_k]$. 

\end{proof}

%%%%%%%%%%%%%%%%%%%%%%%%%%%%%%%%%%%%%%%%%%%%
%%%%%%%%%%%%%%%%%%%%%%%%%%%%%%%%%%%%%%%%%%%%%

\section{Approximate solutions}

In this section we define the class of the approximate solutions ${\bf p}(x,{\bf b})$ in $\R^2$ which are perturbations of the one-dimensional profile ${\bf p}(x_2)$ with ${\bf p} \in \mathcal P^c$, and collect some of their properties. We establish the algebraic statement that the error in the Euler-Lagrange equation cannot be improved further unless ${\bf p}(x,{\bf b})$ is a rotation of ${\bf p}$, see Lemma \ref{eb}. In Corollary \ref{C1}, we obtain the convergence of the rescaled errors between ${\bf u}$ and an approximate solution ${\bf p}(x,{\bf b})$.

\

We begin with the definition of the approximate solution ${\bf p}(x,{\bf b})$. 
\begin{defi}\label{pb}
Given ${\bf p} \in \mathcal P^c$ and ${\bf b} \in B({\bf p})$, we denote by 
$${\bf p}(x, {\bf b}) = {\bf h}\left(x_2, x_1{\bf b}\right).$$
\end{defi}  
Clearly ${\bf p}(x,{\bf b})$ is a homogenous of degree 2 function in its variables.

\begin{lem}\label{L31}
${\bf v}(x):={\bf p}(x, {\bf b}) $ satisfies

a) it solves the Euler-Lagrange equations with error $C \|{\bf b}\|^2$. Precisely ${\bf v} \in C^{1,1},$
$v_1 \ge .. \ge v_N$ and in an open region where $v_i > v_{i+1}$ and $v_k > v_{k+1}$ we have
$$ |\triangle v_I - f_I| \le C |{\bf b}|^2 , \quad \quad \mbox{with} \quad I=\{i+1,..,k\}.$$

b) $$v_i(x)=p_i(x_2)+b_i x_1 x_2 + O(|{\bf b}|^2 x_1^2),$$
and in the cone $\{ |x_2| \ge C \|{\bf b}\| |x_1|\} $ with $C$ large universal 
$$ \triangle v_i= \triangle p_i + 2e_i({\bf b})\chi_{\{x_1 \ge 0\}} +  2e_i(-{\bf b})\chi_{\{x_1 \le 0\}}.$$
where ${\bf e}({\bf b})$ is the error function defined in Definition \ref{ER}.
\end{lem}

\begin{proof} By definition $v$ solves the Euler-Lagrange equations in the $x_2$ variable hence
$$ \triangle v_I - f_I = \partial_{x_1 x_1} v_I.$$
Using the homogeneity of ${\bf h}$ we find
$$ {\bf v}_{11}= 2 {\bf h} - 2 t {\bf h}_t + t^2 {\bf h}_{tt},$$
where $h$ and its derivatives are evaluated at $(t, \frac{x_1}{|x_1|} \, {\bf b})$ with $t:=x_2/|x_1|$. 

Moreover, by Lemma \ref{L2.2}, the right hand side is constant in each of the 4 connected regions of the set $\{|x_2| > C \|{\bf b}\| |x_1|\} \setminus \{x_1=0\}$ and equals $$  {\bf v}_{11}= 2 {\bf e} ({\bf b})\chi_{\{x_1 \ge 0\}} +  2 {\bf e}(-{\bf b})\chi_{\{x_1 \le 0\}}.$$
\end{proof}

\begin{defi}\label{pbb}
Similarly we may define the more general class of functions 
$${\bf p}(x, {\bf b}_0,{\bf b}_1) = {\bf h}\left(x_2, {\bf b_0} + x_1{\bf b}_1\right).$$
When ${\bf b}_0=0$ we are in the situation of Definition \ref{pb} and then use the simpler notation ${\bf p}(x,{\bf b}_1)$ for ${\bf p}(x, {\bf 0},{\bf b}_1)$ as before.  
\end{defi}

We give the corresponding lemma for this more general class of solutions.
\begin{lem}\label{L32}
${\bf v}(x):={\bf p}(x, {\bf b}_0,{\bf b}_1) $ satisfies

a) it solves the Euler-Lagrange equations with error $C \|{\bf b}\|^2$. Precisely ${\bf v} \in C^{1,1},$
$v_1 \ge .. \ge v_N$ and in an open region where $v_i > v_{i+1}$ and $v_k > v_{k+1}$ we have
$$ |\triangle v_I - f_I| \le C |{\bf b}_1|^2 , \quad \quad \mbox{with} \quad I=\{i+1,..,k\}.$$

b) $$v_i(x)=p_i(x_2)+ b_{0,i} x_2 + b_{1,i} x_1 x_2 + O(|{\bf b}_0|^2 + |{\bf b}_1|^2 x_1^2).$$

\end{lem}

\begin{proof}
The proof is the same as above, and follows from $|D^2 {\bf h}| \le C$ (see Lemma \ref{L2.3}) and Lemma \ref{L2.2}.
\end{proof}

\begin{lem}\label{eb}
${\bf e}({\bf b})= {\bf e} (- {\bf b})$ if and only if ${\bf b}=s \tau$ for some $s \in \R$, where $\tau$ is defined in Definition \ref{tau}. 
\end{lem}

Notice that ${\bf b}=s \tau$ is equivalent to ${\bf p}(x,{\bf b})= {\bf p}(x_2 + s x_1)$.

As a consequence of the homogeneity of ${\bf e}$ we can quantify the difference between ${\bf e}({\bf b})$ and ${\bf e} (- {\bf b})$ in terms of the distance from $\bf b$ to the line of direction $\tau$.

\begin{cor}\label{cor31}
There exists a strictly increasing continuous function $$\rho: [0,2] \to [0, \infty), \quad \mbox{with} \quad \rho(0)=0,$$ such that
$$\frac{|{\bf e}({\bf b})- {\bf e} (- {\bf b})|}{ \|{\bf b}\|^{2}} \ge  \rho \left (dist \left (\frac {{\bf b}}{\| {\bf b}\|}, \pm \frac {\tau}{\|\tau\|}\right )  \right), \quad \quad \forall {\bf b} \ne 0.$$
\end{cor}

Since $e({\bf b})$ is piecewise quadratic in ${\bf b}$ it follows that $\rho(s) \ge c s^2$.

\

{\it Proof of Lemma \ref{eb}.} One implication is trivial. 

Due to the homogeneity of ${\bf e}$ it suffices to assume that ${\bf e}({\bf b})= {\bf e} (- {\bf b})$ and $\|{\bf b}\| \le \delta$ for some small $\delta$ 
universal. Let $\Gamma_i^+$ denote the free boundaries for the $1D$ solution ${\bf h}(t, {\bf b})$ and $\Gamma_i^-$ the free boundaries of  ${\bf h}(t, -
{\bf b})$. We want to show that all $\Gamma_i^+$ coincide and that $\Gamma_i^-=-\Gamma_i^+$.

By the lemma above, the function ${\bf v}(x):={\bf p}(x, {\bf b}) $ is a solution to the problem $P$ with an error $C \delta^2$, in the sense that 

1) $v \in C^{1,1}$, $v_1 \ge .. \ge v_N$, 

2) the free boundaries of ${\bf v}$ are given by the rays $x_2= \Gamma_i^+ x_1$ in $\{x_1>0\}$ and $x_2= -\Gamma_i^- x_1$ in $\{x_1<0\}$,

3) in each of the sectors determined by these rays, the component $v_i$ solves the equation $\triangle v_i= g_I$ with $g_I$ a constant, and $|g_I-f_I| \le C \delta^2$, where $I$ is the set of $j$'s for which $v_j=v_i$ in that sector. 

Notice that ${\bf e}({\bf b})= {\bf e} (- {\bf b})$ is equivalent to the statement that the corresponding right hand sides $g_I$ agree on either side of the $x_2$-axis on the 
two sectors that contain the positive respectively negative $x_2$-axis. Also, if $\delta$ is chosen small then the nondegeneracy condition holds for the 
right hand sides $g$, i.e. $\triangle v_i > \triangle v_k$ if $v_i > v_k$.
Now we can argue as in the classification of homogenous solutions in 2D to conclude that all free boundaries coincide with a single line passing 
through the origin, which gives the desired conclusion. We provide the details.

 We denote by $(r,\theta)$ the polar coordinates in $\R^2$. Recall the following elementary lemma from \cite{SY1}:
 
 \begin{lem}\label{c1}
Assume $w$ is homogenous of degree 2 and is defined in the angle $\theta \in [0,\alpha]$ with $w=0$, $\nabla w =0$ on the rays $\theta=0$, $\theta = \alpha$. If $$\triangle w=\varphi \ge 0$$
and $\varphi$ is a step function which is nondecreasing in $[0,\gamma]$, and nonincreasing in $[\gamma, \alpha]$ for some $\gamma$, then $$\alpha \ge \pi.$$ Moreover, if $\alpha=\pi$ then $\varphi$ must be constant. 
\end{lem}

We restrict our attention to the values of $v_i$ on the unit circle $\partial B_1$. We know that each two consecutive membranes $v_i$ and $v_{i+1}$ are connected (agree) at least on an open interval that contains either $(0,1)$ or $(0,-1)$, and they do not agree on the whole circle. 

We focus on those intervals $I \subset \partial B_1$ where $\{v_k>v_{k+1}\}$ and $v_k=v_{k+1}$ at the end points and in addition $\triangle v_k$ is constant in $I$.

{\it Claim:} Each such interval has length greater than or equal to $\pi$.

Indeed, we look at a minimal such interval and we apply Lemma \ref{c1} to the difference  
$$w_k := v_k - v_{k+1},$$
which vanishes of order two at the end points of $I$. Moreover, 
$$\triangle w_k = \varphi_k:=g_k-g_{k+1}  >0  \quad \mbox{on $I$,} $$
The minimality of $I$ implies that the nested sets $\{ v_{k+1}=v_{k+m} \}$ are connected (intervals) in $I$, and therefore $w_k$, $\varphi_k$ satisfy the hypotheses of the Lemma \ref{c1}.

\

The claim implies that $\{v_1>v_2\}$ consists of exactly one interval $I_1$ of length at least $\pi$. 
In the cone generated by $I_1$, the function $v_1$ coincides with a quadratic polynomial $Q$. Denote by $\tilde v_1$ this polynomial $Q$ in the complement of the angle  generated by $I_1$. Here we can apply one more time the argument of Claim above by using the function $\tilde w_1:= \tilde v_1-v_2$ and conclude that also the complement has length at least $\pi$ on the unit circle. 

In conclusion $I_1$ consists exactly of a half-circle. 
Lemma \ref{c1} gives in addition that $\triangle v_2$ is in fact constant on $I_1$ and its complement. This in turn implies that $v_2$ and $v_3$ either 
coincide or are disjoint in each of these two intervals. By arguing as above with $v_2$ and $v_3$, instead of $v_1$ and $v_2$ we find that also $\triangle v_3$ must be 
constant in each of these two intervals, which gives that $\{v_3 =v_4\}$ is either $I_1$ or its complement. 
We can argue like this inductively and reach that all the free boundaries must coincide.

\qed

\begin{defi}\label{S}
Given ${\bf p} \in \mathcal P^c$, we say that a solution $\bf u$ to the problem $P_0$ is $\eps$- approximated in $B_r$ and write
$$ {\bf u} \in \mathcal S (r,{\bf p}, \eps )$$ 
if, after a rotation around the origin, ${\bf u}$ satisfies
$$ |{\bf u} - {\bf p}(\cdot, {\bf b})| \le \eps r^2 \quad \mbox{in $B_r$}, $$
$$\mbox{for some} \quad {\bf b} \in B({\bf p}), \quad \mbox{with} \quad |{\bf b}| \le \delta \eps^{1/2},$$
with $\delta$ a small universal constant (to be made precise later).
\end{defi}

\begin{lem}\label{l300}
Assume that 
\begin{equation}\label{e30}
 {\bf u} \in \mathcal S (1,{\bf p}, \eps ) .
 \end{equation}
 Then in $B_{3/4}$ we have $\Gamma_i \subset \{|x_2| \le C \sqrt \eps\}$ for all $i$, and
 \begin{equation}\label{e32}
|\triangle (u_i - p_i(\cdot, {\bf b}))| \le \delta \eps \quad \quad \mbox{in} \quad \{|x_2| \ge C \sqrt \eps\} \cap B_{3/4}. 
\end{equation}
\end{lem}

\begin{proof}
Any two consecutive membranes, say $u_i$ and $u_{i+1}$, coincide on one side of this strip $\{|x_2| \le C \sqrt \eps\}$ 
and are separated on the opposite side, 
depending on whether the membranes $p_i$ and $p_{i+1}$ of the 1D- solution ${\bf p} \in \mathcal P^c$ 
coincide to the right or left of the origin. 

Indeed, assume that $p_i=p_{i+1}$ to the left of the origin, and then 
$$p_i(x,{\bf b})-p_{i+1}(x,{\bf b}) \ge c \left[(x_2 - C |x_1{\bf b}|)^+ \right]^2,$$
$$p_i(x,{\bf b})=p_{i+1}(x,{\bf b}) \quad \mbox{ if} \quad  x_2 \le - C |x_1{\bf b}|.$$
The bound $|{\bf b}| \le \delta \sqrt \eps$ from Definition \ref{S} and \eqref{e30} implies that 
 $$u_i > u_{i+1} \quad \mbox{ in} \quad  B_1 \cap \{x_2 \ge C \sqrt \eps\},$$ 
 $$|u_i-u_{i+1}| \le 2 \eps \quad \mbox{ in} \quad  B_1 \cap \{x_2 \ge C \sqrt \eps\}.$$ 
 The claim $$\Gamma_i \subset \{|x_2| \le C \sqrt \eps\} \cap B_{1-C \sqrt \eps},$$ follows since $u_i$ and $u_{i+1}$ separate quadratically away from 
 their free boundary $ \Gamma_i$.

As a consequence we find that in $\{|x_2| \ge C \sqrt \eps\} \cap B_{3/4} $,
$$\triangle u_i= \triangle (p_i(x_2))=f_I \quad \mbox{ in} \quad \{|x_2| \ge C \sqrt \eps\},$$
and, by Lemma \ref{L31}, 
\begin{equation*}\label{e322}
|\triangle (u_i - p_i(\cdot, {\bf b}))| \le C |{\bf b}|^2 \le C \delta^2 \eps \le \delta \eps,
\end{equation*}
provided $\delta$ is sufficiently small. 
\end{proof}

\begin{lem}\label{l36}Assume that $ {\bf u} \in \mathcal S (1,{\bf p}, \eps ).$ Then in $B_{1/2}$
$$ |{\bf u} - {\bf p}(\cdot, {\bf b})| \le C \, \eps (|x_2|+ \sqrt \eps)^\alpha,$$
for some $\alpha>0$ small, universal.
\end{lem}

\begin{proof}
We pick a point $Z=(z,0)$, $|z| \le 1/2$ on the $x_1$ axis. It suffices to show by induction that for $k \ge 0$,
$$ |u_i - p_i(\cdot, {\bf b})| \le \eps_k:= \eps (1-c)^{k} \quad \mbox{in} \quad B_{r_k}(Z), \quad r_k:=\rho^{k+1},$$
as long as $r_k \ge C' \sqrt \eps$, where $\rho$, $c$ are small, universal constant.

Assume the induction hypothesis holds for $k$ and suppose that $\bf p$ has at least two branches on the right (in the $x_2$-direction). 
We denote by $Y:=Z+ \frac 12 r_k e_2$, and we claim that if 
\begin{equation}\label{ujy}
u_j(Y) \ge p_j(Y,{\bf b})  \quad \mbox{for some $j$,}
\end{equation}
then
\begin{equation}\label{usubi}
u_i - p_i(\cdot,{\bf b}) \ge (c-1) \eps_k  \quad \mbox{in} \quad B_{\rho r_k}(Z), \quad \forall \, i.
\end{equation}
By Lemma \ref{l300}, we know that 
$$|\triangle (u_i - p_i(\cdot, {\bf b}))| \le \delta \eps \le \delta \eps_k r_k^{-2}, \quad \mbox{in} \quad \{|x_2| \ge C \sqrt \eps\} \cap B_{r_k}(Z),$$
and $$u_i - p_i(\cdot, {\bf b}) \ge - \eps_k \quad \mbox{ in} \quad  B_{r_k}(Z),$$ by the induction hypothesis. 
We prove \eqref{usubi} by comparing $\bf u$ with an explicit subsolution $\bf v$ in the rectangle 
$$R:=\{|x_1-z| \le r_k/2\} \times \{|x_2| \le 4 \rho r_k \}.$$
The Harnack inequality and \eqref{ujy} imply that
\begin{equation}\label{ujs}
u_j - p_j(\cdot, {\bf b}) \ge  (c_0-1) \eps_k \quad \mbox{on} \quad \partial R \cap \{x_2= 4 \rho r_k\},
\end{equation}
for some $c_0=c_0(\rho)$ universal. This inequality holds for all other membranes which coincide with $u_j$ in the region $\{x_2 \ge C \sqrt \eps\}$. We denote by 
$J$ these indexes $l$ for which $u_l(Y)=u_j(Y)$, and remark that $J$ depends only on the branch configuration of ${\bf p}$. We let ${\bf t} \in B({\bf p})$ be defined as $t_i^-=0$ for all $i$, and 
$$\mbox{$t_i^+=1$ if $i \in J$, and $t_i^+=-\mu$ otherwise.}$$ 
The constant $\mu>0$ is chosen such that the 
average of all the $t_i^+$ equals $0$, so that ${\bf t} \in B({\bf p})$.

We define the barrier (see Definition \ref{pbb})
$${\bf v}(x):={\bf p}(x_2,  {\bf d},{\bf b}) + (c_1 \eps_k  q((x-Z)/r_k) -\eps_k) {\bf 1} ,$$
\begin{equation}\label{dq}
{\bf d}:= c_1 \, \eps_k r_k^{-1} {\bf t}, \quad q(x):= \frac \mu 2 (x_2 + 2 \rho) + x_2^2 -\frac 12 x_1^2,
\end{equation}
where $c_1$ is small, depending on the constant $c_0$ above. The polynomial $q$ and the constant $\rho$ are chosen such that $\triangle q =1$,
\begin{equation}\label{qineq}
q + t_i x_2^+\ge c_2:=\frac 12 \mu \rho  \quad \mbox{in $B_\rho$,}   
\end{equation}
and on the boundary of the rescaled rectangle 
$$R_0:= \{|x_1| \le 1/2\} \times \{|x_2| \le 4 \rho \},$$
we have
$$q + t_i x_2^+ \le - c_2 \quad \mbox {on} \quad \partial R_0 \setminus \{x_2= 4 \rho\} ,\quad \forall i,$$
\begin{equation}\label{q+}
q + t_i x_2^+ \le - c_2  \quad \mbox {on} \quad \partial R_0 \quad \mbox{if} \quad i \notin J.
\end{equation}
We check that ${\bf u} \ge {\bf v}$ on $\partial R$, and ${\bf v}$ is a subsolution to the problem $P$. 

By Lemma \ref{L32}, ${\bf p}(x,  {\bf d},{\bf b})$ solves the problem $P$ 
with an error $$C |{\bf b}|^2 \le C \delta^2 \eps \le \delta \eps \le \delta \eps_k r_k^{-2},$$ and since $\triangle q =1$ it follows that ${\bf v}$ is a subsolution to the problem $P$ if $\delta$ is sufficiently small ($\delta \le c_1$).

Notice that $\eps_k r_k^{-2}$ is increasing with $k$, and when $r_k \sim C' \sqrt \eps$ then $$\eps_k r_k^{-2} \le C \eps^\alpha \le \delta^2 \quad \mbox{ provided that $\eps \le \eps_0(\delta)$}.$$ Thus,
$$C |{\bf d}|^2 \le C \eps_k^2 r_k^{-2} \le \delta \eps_k, \quad \quad \mbox{and} \quad C |{\bf b}|^2x_1^2 \le \delta \eps_k,$$
and by Lemmas \ref{L32} part b)
\begin{equation}\label{rkZ}
|p_i(x,  {\bf d}, {\bf b}) - p_i(x,  {\bf b}) - c_1 \eps_k r_k^{-1} \, t_i \, x_2^+| \le 3 \delta \eps_k \quad \mbox{in} \quad B_{r_k}(Z).
\end{equation}
Using the inequalities \eqref{q+} of $q$ on $\partial R_0$ we obtain that   
$$ v_i \le p_i(\cdot,  {\bf b}) + \eps_k ( 3 \delta - c_1 \, c_2 -1) \le p_i(x,  {\bf b}) -\eps_k \le u_i \quad \mbox {on} \quad \partial R \quad \mbox{if} \quad i \notin J,$$
$$v_i \le p_i(x,  {\bf b}) -\eps_k \le u_i \quad \mbox {on} \quad \partial R \setminus \{x_2= 4 \rho r_k\} ,\quad \forall i.$$
Finally, on $\partial R \cap \{x_2= 4 \rho r_k\}$ and $i \in J$ we have by \eqref{ujs}
$$v_i \le p_i(x,  {\bf b}) + (C(\mu,\rho) c_1-1) \eps_k \le u_i,$$
provided that $c_1$ is chosen small so that $C(\mu,\rho) c_1 \le c_0$. 

In conclusion ${\bf u} \ge {\bf v}$ on $\partial R$, and the inequality holds in the whole $R$ by the maximum principle. In particular, by \eqref{qineq} in $B_{\rho r_k}$
$$ u_i \ge v_i \ge p_i(\cdot, {\bf b}) + (-3 \delta + c_1 c_2 -1) \eps_k \ge p_i(\cdot, {\bf b}) + (c -1) \eps_k.$$
\end{proof}

\begin{cor}\label{C1}
If $ {\bf u}_m \in \mathcal S (1,{\bf p}, \eps_m ),$ for a sequence of $\eps_m \to 0$, then, up to a subsequence, then each of the rescaled error functions
$$\eps_m^{-1} \left(u_{m,j} - p_j(\cdot,{\bf b}_m )\right)$$ converges uniformly in $B_{1/2}$ to a limit $w_j$ that satisfies 
$$\|w_j\|_{L^\infty} \le 1, \quad w_j =0  \quad \mbox{on $x_2=0$,}$$
and $$|\triangle w_j| \le \delta \quad \mbox{away from $\{x_2=0\}$.}$$
More precisely, $\triangle w_j$ is constant in each quadrant   
$$\triangle w_j = -2e_j({\bf b}) \chi_{\{x_1 >0 \}} -  2e_j(-{\bf b}) \chi_{\{x_1 <0 \}}\quad \mbox{in $\{x_2<0\} \cup \{x_2>0\}$,}$$
where ${\bf b} \in B({\bf p})$ is the limit of  
$$ {\bf b}:= \lim_{m \to \infty} \eps_m^{-\frac 12}{\bf b}_m, \quad \quad |{\bf b}| \le \delta.$$
\end{cor}

\begin{proof}
The convergence to a limit $w_j$ as above follows directly from Lemmas \ref{l300} and \ref{l36}. The second part is a consequence of $|{\bf b}_k| \le \delta \eps_k^{1/2}$ (see Definition \ref{S}), and Lemma \ref{L31} 
part b), after recalling that the function ${\bf e}({\bf b})$ is homogenous of degree 2 in $\bf b$ (see Definition \ref{ER}). 

\end{proof}

\section{Weiss monotonicity}

In this section we establish the upper bound for the Weiss energy in Lemma \ref{L41} and the main dichotomy result Proposition \ref{PM}, which give Theorem \ref{TIntro1} in the case of non-degenerate cones.

\

We denote by 
$$E({\bf u},r):= r^ {-(n+2)} \int_{B_r}\sum \,  \omega_k(\frac{1}{2}|\nabla u_k|^2+f_ku_k) \, dx$$
and
$$F({\bf u},r):= r^{-(n+3)}\int_{\partial B_r} \sum \omega_k u_k^2 \, d \sigma.$$
The Weiss functional is
$$ 
W({\bf u},r):= E({\bf u},r)- F({\bf u},r).$$

We compute
\begin{align*} \frac{d}{dr} W({\bf u},r) = r^ {-(n+2)} \int_{\partial B_r}\sum& \,  \omega_k \left(\frac{1}{2}|\nabla u_k|^2+f_ku_k - 2 r^{-1}u_k u_{k,\nu} + 4 r^{-2}u_k^2 \right) \, dx \\
& \quad \quad - (n+2) r^{-1} E({\bf u},r) \\
= r^ {-(n+2)} \int_{\partial B_r}\sum& \,  \frac {\omega_k}{2} (u_{k,\nu} - \frac 2 r u_k)^2 \, d \sigma + \frac {n+2} {r} (E({\bf u}_h,r) - E({\bf u},r)) \\
\ge r^ {-(n+2)} \int_{\partial B_r}\sum &\, \frac {\omega_k}{2} (u_{k,\nu} - \frac 2 r u_k)^2 \, d \sigma,
\end{align*}
where ${\bf u}_h$ denotes the homogenous of degree 2 extension of the boundary data of ${\bf u}$ on $\partial B_r$, and in the last inequality we used the minimality of ${\bf u}$ for the energy $E$ in $B_r$.

\begin{lem}\label{L41}
Assume that $ {\bf u} \in \mathcal S (1,{\bf p}, \eps ).$ Then
$$W({\bf u},1/2) \le W({\bf p}) +  C \eps^{3/2}.$$
\end{lem}

\begin{proof}
We denote by ${\bf v}:={\bf p}(\cdot,{\bf b})$ and we prove the following inequalities
\begin{equation}\label{4100}
W({\bf u},1/2) \le W({\bf v}) + C \eps^2,
\end{equation}
and
\begin{equation}\label{4101}
W({\bf v}) \le W({\bf p}) + C \eps^{3/2}.
\end{equation}
In order to obtain \eqref{4100} we write
$${\bf v}={\bf u} + \eps {\bf w}, \quad \quad |{\bf w}| \le 1.$$
By Lemmas \ref{l300}, \ref{l36} we know that outside the strip $\{ |x_2| \le C \sqrt \eps\}$ each component $w_k$ satisfies $|\triangle w_k|\le \delta$, hence
\begin{equation}\label{nabw}
|\nabla {\bf w}| \le C(|x_2| + \sqrt \eps) ^{\alpha -1}, \quad \mbox{in} \quad \{ |x_2| \ge C \sqrt \eps\}\cap B_{1/2}.
\end{equation}
Inside the strip, the $C^{1,1}$ norm of $w_k$ is bounded by $C \eps^{-1}$, hence 
\begin{equation}\label{nabw2}
|\nabla {\bf w}| \le C \eps^{-1/2} \quad \mbox{in} \quad \{ |x_2| \le C \sqrt \eps\}\cap B_{1/2}.
\end{equation}
Then, with $r=1/2$, we write
$$W({\bf v}, r) = W({\bf u},r) + \eps^2 r^{n-2} I_1 +  \eps r^{n-2}  I_2,$$
with
$$I_1:= \int_{B_r}\sum \,  \frac{\omega_k}{2}|\nabla w_k|^2 dx - r^{-1} \int_{\partial B_r} \sum \omega_k w_k^2 \, d \sigma,$$
\begin{align*}
I_2:=&\int_{B_{r}} \sum \omega_k (\nabla u_k \cdot \nabla w_k + f_k w_k) dx - 
\int_{\partial B_r} \sum\omega_k  \, \frac 2 r u_k w_k d \sigma\\
= &\int_{B_{r}} \sum \omega_k (f_k -\triangle u_k) w_k dx + \int_{\partial B_r} \sum  \omega_k (u_{k,\nu} - \frac 2 r u_k) w_k d \sigma\\
\ge & \quad \eps \quad \int_{\partial B_r} \sum  \omega_k (-w_{k,\nu} + \frac 2 r w_k) w_k d \sigma.
\end{align*}
In the last inequality we used (see \eqref{EL1}) 
\begin{equation}\label{uineq}
\sum \omega_k (f_k -\triangle u_k) w_k \ge 0,
\end{equation} and that ${\bf v}$ is homogenous of degree 2.
From \eqref{nabw}-\eqref{nabw2} we infer that $I_2 \ge - C \eps$. Since $I_1 \ge -C$ we conclude that \eqref{4100} holds.

For the second inequality \eqref{4101} we argue similarly. We denote
$${\bf p}={\bf v} + {\bf g},$$
for some ${\bf g}$ that satisfies (see Lemma \ref{L31} part b))
$$|{\bf g}| \le C \sqrt \eps \quad \mbox{in $B_1$,} \quad \quad |{\bf g}| \le C \eps \quad \mbox{in $\{ |x_2| \le C \sqrt \eps\} \cap B_1$.}$$
We have
$$W({\bf p}) = W({\bf v}) + I_3$$
with
\begin{align*}
I_3:=&\int_{B_{1}} \sum \omega_k \left(\nabla v_k \cdot \nabla g_k + \frac 12 |\nabla g_k|^2 + f_k g_k \right) dx - \int_{\partial B_1} \sum  \omega_k  \,(2 v_k g_k + g_k^2) d \sigma\\
=& \int_{B_1}  \sum \omega_k \left(f_k - \triangle v_k - \frac 12 \triangle g_k \right) g_k \, dx,
\end{align*}
where we have used that ${\bf v}$ and ${\bf g}$ are homogenous of degree 2. 

We estimate the last integral.
When $x$ belongs to the strip
 $ \{|x_2| \le C \sqrt \eps\}$ then 
 $$|g_k| \le C \eps \quad \mbox{and} \quad  |f_k - \triangle v_k - \frac 12 \triangle g_k| \le C,$$
 while outside the strip we have (see Lemma \ref{L31} part a) and Lemma \ref{l300})
$$| \sum \omega_k \left(f_k - \triangle v_k \right) g_k|\le C \eps ^{3/2}, \quad \quad |\triangle g_k| \le \eps.$$
Thus $|I_3| \le C \eps^{3/2}$, and \eqref{4101} is proved.

\end{proof}

\begin{prop}\label{PM}
Assume that $ {\bf u} \in \mathcal S (1,{\bf p}, \eps ),$ with $\eps \le \eps_0$. Then either
$$ {\bf u} \in \mathcal S (\rho,{\bf p}, \frac \eps 2),$$
or
$$ {\bf u} \in \mathcal S (\rho,{\bf p}, C \eps ), \quad \mbox{and} \quad W({\bf u}, \rho) \le W({\bf u},1) - c  \, \eps^2.$$
Here $\rho$, $\eps_0$, $c$ (small) and $C$ (large) denote universal constants.
\end{prop}

\begin{rem}
If ${\bf v}_1$ and ${\bf v}_\rho$ denote the approximate solutions of the type ${\bf p}(\cdot, {\bf b})$ in $B_1$ respectively $B_\rho$, that appear in the conclusion of Proposition \eqref{PM},  (see Definition \ref{S}), then they must be $C \eps$-close to each other i.e. 
$$\|{\bf v}_1-{\bf v}_\rho\|_{L^\infty(B_1)}\le C \eps.$$
\end{rem}

\begin{proof} We remark that the first conclusion of the second alternative ${\bf u} \in \mathcal S (\rho,{\bf p}, C \eps )$ is obvious, by taking $C = \rho^{-2}$.

We prove the statement by compactness. We fix $\rho=1/4$, $C=\rho^{-2}$, and assume that there exists a sequence of ${\bf u}_m$, ${\bf b}_m$, $\eps_m \to 0$ for which the conclusion 
does not hold with $c_m = 1/m \to 0$. 
By Corollary \ref{C1} we may extract a subsequence of the rescaled errors
$${\bf w}_m:=\eps_m^{-1} \left({\bf u}_{m} - {\bf p}(\cdot, {\bf b}_m )\right)$$ 
which converges uniformly in $B_{1/2}$ (and in $C^1_{loc}(B_{1/2} \setminus\{ x_2=0\})$) to a limit function ${\bf w}$ which satisfies
$$ w_j=0 \quad \mbox{on} \quad \{x_2=0\},$$
$$\triangle w_j =-2e_j({\bf b}) \chi_{\{x_1 >0 \}} -  2e_j(-{\bf b}) \chi_{\{x_1 <0 \}}\quad \mbox{in $\{x_2<0\} \cup \{x_2>0\}$,}$$
where ${\bf b} \in B({\bf p})$ is the limit of  
$$ {\bf b}:= \lim_{m \to \infty} \eps_m^{-\frac 12}{\bf b}_m, \quad \quad |{\bf b}| \le \delta.$$
Since
\begin{align*}
\eps_m^{-2}(W({\bf u}_m, 1)-W({\bf u}_m, \rho) ) &= \eps_m^{-2} \int_\rho^1 \frac {d}{dr} W({\bf u}_m, r) dr \\
& \ge \int_{B_1 \setminus B_\rho} r^ {-(n+2)} \sum \, \frac {\omega_k}{2} (\partial_\nu w_{m,k} - \frac 2 r w_{m,k})^2 \, d \sigma, 
\end{align*}
we may take $m \to \infty$ and conclude that ${\bf w}$ is homogenous of degree 2 in $B_{1/2}$ (first in $B_{1/2} \setminus B_\rho$ by the inequality above, and then in $B_{1/2}$ by unique continuation). This implies that ${\bf e}({\bf b})={\bf e}(-{\bf b})$ and by Lemma \ref{eb} we conclude that 
\begin{equation}\label{betas}
{\bf b}=s \tau \quad \mbox{ for some} \quad s \in [-C \delta, C\delta].
\end{equation}
Moreover,
$$ w_j:= \gamma_j \,  x^2_2 + \left (t_j^+  \chi_{\{x_2>0\}} + t_j^- \chi_{\{x_2<0\}}\right) x_1 x_2,$$
with $\gamma_j=-e_j({\bf b})$, and
$$|{\bf \gamma}| = |{\bf e}({\bf b})| \le C |{\bf b}|^2 \le C \delta^2 \le \delta.$$
Moreover, since the average of $w_j$ is $0$ then ${\bf t} \in B({\bf p})$, $|{\bf t}| \le C$.
Using Lemma \ref{L31} part b), we find that
$${\bf p} (\cdot, {\bf b}_m + \eps_m {\bf t} )= {\bf p} (\cdot, {\bf b}_m) + \eps_m {\bf w} - \eps_m \, x_2^2 \,  \gamma  + O\left((|{\bf b}_m|^2 + |{\bf b}_m+ \eps_m {\bf t}|^2)  x_1^2\right)$$
hence
\begin{equation}\label{usubm2}
|{\bf u}_m - {\bf p} (\cdot, {\bf b}_m + \eps_m {\bf t})| \le  \eps_m ( \delta + C \delta^2) \rho^2 \le \frac{\eps_m}{4} \rho^2 \quad \mbox{in} \quad B_\rho.
\end{equation}
We cannot yet conclude that $ {\bf u}_m \in \mathcal S (\rho,{\bf p}, \eps_m/ 2),$ and reach a contradiction since we do not know that 
$$|{\bf b}_m + \eps_m {\bf t}| \le \delta (\eps_m/2)^{1/2}.$$
We achieve this after a rotation of coordinates. 
We use \eqref{betas} and write $${\bf b}_m + \eps_m {\bf t}=\eps_m^{1/2} (s \tau + {\bf d}_m) \quad \mbox{ with} \quad  {\bf d}_m \to 0,$$ and find (see Definition \ref{tau})
\begin{align*}
{\bf p} (x, {\bf b}_m + \eps_m {\bf t})&={\bf h} (x_2, x_1\eps_m^{1/2} (s \tau + {\bf d}_m))\\
& = {\bf h} (x_2+\eps_m^{1/2}s x_1, x_1\eps_m^{1/2} {\bf d}_m).
\end{align*}
Denote by $(y_1,y_2)$ the new coordinates in the rotated system 
$$y_1:=  (1+ \eps_m s^2)^{-1}(x_1 - \eps_m^{1/2} s x_2), \quad  y_2:= (1+ \eps_m s^2)^{-1}(x_2+\eps_m^{1/2}s x_1),$$
and notice that
$$x_2+\eps_m^{1/2}s x_1=y_2 + O( \eps_m s^2|y|), \quad x_1\eps_m^{1/2} {\bf d}_m=y_1 \eps_m^{1/2} {\bf d}_m + O( \eps_m s|y|).$$
Thus, since ${\bf h}$ is homogenous of degree 2 and has bounded second derivatives,
\begin{align}\label{change}
\nonumber {\bf p} (x, {\bf b}_m + \eps_m {\bf t})=&{\bf h}\left(y_2 + O( \eps_m s^2|y|), y_1 \eps_m^{1/2} {\bf d}_m + O( \eps_m s|y|)\right)\\
\nonumber =&{\bf h}\left(y_2 , y_1 \eps_m^{1/2} {\bf d}_m\right) + O( \eps_m s|y|^2)\\
=& {\bf p}(y,\eps_m^{1/2} {\bf d}_m) + O( \eps_m s|y|^2).
\end{align}
The error term is bounded by (see \eqref{betas}) 
$$|O(\eps_m s|y|^2)| \le C \delta \eps_m |y|^2 \le \frac{\eps_m}{4} |y|^2$$
provided that $\delta$ is chosen small. Also, for all large $m$,
$$|\eps_m^{1/2} {\bf d}_m|\le \delta (\eps_m/2)^{1/2},$$
and by \eqref{usubm2} we conclude $ {\bf u}_m \in \mathcal S (\rho,{\bf p}, \eps_m /2),$ which is a contradiction.

\end{proof}

\begin{thm}\label{Td2}
Assume that $d=2$ and ${\bf p} \in \mathcal P^c$ is a blow-up limit for $\bf u$ at the origin. Then, $\bf p$ is unique and
$${\bf u}(x)={\bf p}(x_2) + O(|x|^2(-\log |x|)^{-1}).$$
\end{thm}
\begin{proof}
It follows from Lemma \ref{L41}, Proposition \ref{PM} and Lemma \ref{Sequences}. We omit the details.
\end{proof}

%%%%%%%%%%%%%%%%%%%%%%%%%%%%%%%%%%%%%%%%%%%%%%
%%%%%%%%%%%%%%%%%%%%%%%%%%%%%%%%%%%%%%%%%%%%%%

\section{The degenerate cones}

In this section we prove Theorem \ref{TIntro1} for degenerate 2D cones. The main ideas are similar to the ones of the previous section, however the convergence of the rescaled errors is much more delicate in this case. Also the compactness argument is more involved due to the geometry of singular cones.  

\

We consider 1D cones which do not belong to $\mathcal P^c$, and their two-dimensional analogues. 
Fix such a one-dimensional cone $${\bf p}_* \in \mathcal P \setminus \mathcal P^c.$$ We can decompose ${\bf p}_*$ as a union of $m \ge 2$ cones in $\mathcal P^c$ as follows. 

Let $k_1 < k_2 < ..< k_{m-1}$ be the indices $k$ with trivial coincidence sets, i.e. 
$$\{p_{*,k}=p_{*,k+1}\} =\{0\}.$$ The consecutive membranes in each of the $m$ groups 
$\{p_{*,k_i}, p_{*,k_i+1},..,p_{*,k_{i+1}-1}\}$ are connected nontrivially on a half-line. After subtracting the average $q_{*,i}$ (a quadratic polynomial) from each group we define the corresponding vector
$${\bf p}^i_* := (p_{*,k_{i-1}+1},..,p_{*,k_i})-(q_{*,i},q_{*,i},...,q_{*,i}): \R \to\R^{k_{i}-k_{i-1}},$$
and ${\bf p}^i_*$ is a connected cone for the $k_i-k_{i-1}$ membranes.
Thus we can write ${\bf p}_*$ as a union of $m$ connected cones
\begin{equation}\label{p*}
{\bf p}_*=({\bf p}^1_* + q_{*,1} {\bf 1}, ...,   {\bf p}^m_* + q_{*,m} {\bf 1}), \quad \quad {\bf p}^i_* \in \mathcal P^c.
\end{equation}
The analogue cones in 2D corresponding to ${\bf p}_*$ have the form
\begin{equation}\label{degp}
{\bf p}=({\bf p}^1 + q_{1} {\bf 1}, ...,   {\bf p}^m + q_{m} {\bf 1}), 
\end{equation}
with $q_i$ quadratic polynomials such that $$\triangle q_i=q_{*,i}'', \quad \quad \sum \omega_k q_i=0,$$ and with ${\bf p}^i$ obtained from ${\bf p}_*^i$ after a rotation. Here ${\bf p}_*^i$ represents the trivial extension from 1D to 2D while the angle of rotation depends on $i$. The polynomials $q_i$ and rotations ${\bf p}^i$ are constrained by the condition $p_{k} \ge p_{k+1}$ which must hold for all $k \ge 1$. This condition needs to be checked only for consecutive membranes  belonging to different connected groups, i.e. when $k$ is one of the $k_i$'s, since it is clearly satisfied within each connected group.
 
When ${\bf p} \in \mathcal C_2$ is a 2D-cone extension of ${\bf p}_*$ as in \eqref{degp} we write
$$ {\bf p} \in \mathcal P({\bf p}_*).$$
For such a cone ${\bf p}$, the free boundaries $$\Gamma_k:=\partial \{p_k >p_{k+1} \}$$ with $k_{i-1} <k<k_i$ 
coincide with a single line, the line of the rotation of ${\bf p}^i_*$ (whenever ${\bf p}^i_*$ consists of at least two membranes). 
When $k=k_i$ then the free boundary $\Gamma_k$ is the same as the coincidence set $\{p_k=p_{k+1}\}$, and we show that it 
is either the origin, one ray, or two rays passing through the origin. We make this more precise.

\begin{lem}\label{2ray}
$\Gamma_{k_i}$ consists of at most two rays that make an angle strictly greater than $\pi /2$.
\end{lem}

\begin{proof}

Lemma \ref{c1} which implies that in each half-plane where $\triangle p_{k_i}$ is constant (or where $\triangle p_{k_i+1}$ is constant), the coincidence 
set cannot contain two distinct rays, unless they coincide with the boundary of the half-plane and both $\triangle p_{k_i}$, $\triangle p_{k_i+1}$ are 
constant on either side of the line. 

This proves that there are at most 2 rays in $\Gamma_{k_i}$. 

Next we denote by $\varphi_j$ the multiplicity 1 parts of  $p_{k_i}$ and $p_{k_i+1}$:
$$p_{k_i}= \varphi_1 + a_1 [(x \cdot \nu_1)^+]^2, \quad \quad p_{k_i+1}= \varphi_2 - a_2 [(x \cdot \nu_2)^+]^2,$$
with $\varphi_j$ homogenous quadratic polynomials, and the constants $a_j \ge 0$. Moreover, by non-degeneracy 
$$\triangle \varphi_1 = f_{k_i} > f_{k_i+1} =\triangle \varphi_2.$$
The coincidence rays are the ones along which $\varphi_2 - \varphi_1$ is tangent by below to the piecewise quadratic function
$$a_1 [(x \cdot \nu_1)^+]^2 + a_2 [(x \cdot \nu_2)^+]^2 \ge 0.$$
If there are two coincidence rays, they must belong to the two different components of $\{\varphi_2 - \varphi_1 >0\}$. 
The conclusion follows since $\varphi_2 - \varphi_1$ is a strictly superharmonic homogenous quadratic polynomial.

\end{proof}

We prove Theorem \ref{Td2} for the degenerate cones.  
\begin{thm}\label{Td2_sec}
Assume that $d=2$ and ${\bf p} \in \mathcal P({\bf p}_*)$ is a blow-up limit for $\bf u$ at the origin. Then, $\bf p$ is unique and
$${\bf u}(x):={\bf p}(x) + O(|x|^2(-\log |x|)^{-1}).$$
\end{thm}

The strategy of proof is the same as in the Sections 3 and 4. First we introduce a family of approximate solutions near cones ${\bf p} \in \mathcal P({\bf p}_*)$ similar to Definition \ref{S}. In this case, an approximate solution ${\bf v}$ consists of a collection of vector-functions ${\bf v}^i$ as in Section 3, with each of them approximating a connected group of ${\bf p}$. More precisely ${\bf v}$ has the form
\begin{equation}\label{bfv}
{\bf v}=({\bf v}^1,..,{\bf v}^m), \quad \quad v_k \ge v_{k+1} \quad \forall k, \quad \quad \sum \omega_k v_k=0,
\end{equation}
$${\bf v}^i={\bf p}^i(x, {\bf b}_i) + q_{i} {\bf 1}, \quad \quad \quad |{\bf b}_i|\le \delta \eps^{1/2},$$ 
with $q_i$ quadratic polynomials with $\triangle q_i=q_{*,i}''$, $\sum \omega_k q_i=0$ and ${\bf p}^i(\cdot, {\bf b}_i)$ represents an $\eps$-approximation of a rotation of the connected 1D cone ${\bf p}^i_*$, as in Definition \ref{pb}.

We make precise the definition of the solutions ${\bf u}$ which can be approximated by such ${\bf v}$'s.

\begin{defi}\label{S_sec}
Given a 1D cone ${\bf p}_*$ as in \eqref{p*}, we say that a solution $\bf u$ to the problem $P_0$ is $\eps$- approximated in $B_r$ by ${\bf p}_*$ and write
$$ {\bf u} \in \mathcal S (r,{\bf p}_*, \eps )$$ 
if, there exists an admissible ${\bf v}$ as in \eqref{bfv} above such that
$$ |{\bf u} - {\bf v}| \le \eps r^2 \quad \mbox{in $B_r$}, \quad \quad |{\bf b}_i|\le \delta \eps^{1/2},$$
with $\delta$ a small universal constant (to be made precise later).
\end{defi}

By definition, ${\bf v} \in C^{1,1}$ is homogenous of degree 2, and the coincidence set between consecutive connected groups i.e. $\{ v_k=v_{k+1}\}$ with $k=k_i$ has empty interior in $\R^2$, since $\triangle (v_k - v_{k+1}) \ge c >0$. Moreover, on the unit circle this difference grows quadratically away from its minimum points, hence the set where $v_k$ and $v_{k+1}$ are $\eps$ close to each other in $B_1$
$$D_k^\eps:=\{v_k - v_{k+1} \le 2 \eps \}\cap B_1$$ 
is included in a $C \eps^{1/2}$-neighborhood of at most 2 rays passing through the origin. The upper bound on the number of rays follows by compactness, since ${\bf v}$ must converge to an element ${\bf p} \in \mathcal P({\bf p}_*)$ as $\eps \to 0$.

By Lemma \ref{L31} part a), ${\bf v}$ satisfies the Euler-Lagrange equations with $\delta \eps$-error
$$|\triangle v_I - f_I| \le C \delta^2 \eps \le \delta \eps.$$
Moreover, if $\nu_i$ denotes the unit direction of rotation for ${\bf p}^i$, so that ${\bf p}^i(\cdot, {\bf b}_i)$ is the $\eps$-approximation of ${\bf p}^i_*(x \cdot \nu_i)$ then,  by Lemma \ref{L31} part b), in $B_1 \cap \{|x \cdot \nu_i| \ge \eps^{1/2}\}$ we have
\begin{equation}\label{bfvi}
\triangle {\bf v}^i= q_{*,i}''+ \triangle {\bf p}_*^i + 2 {\bf e}({\bf b}_i)\chi_{\{x \cdot \nu_i^\perp \ge 0\}} +  2{\bf e}(-{\bf b}_i)\chi_{\{x \cdot \nu_i^\perp \le 0\}}.
\end{equation}

If a solution $ {\bf u} $ is $\eps$-approximated by ${\bf v}$ in $B_1$, then in $B_{1- C \eps^{1/2}}$ the coincidence sets for ${\bf u}$ and ${\bf v}$ agree away from the set 
\begin{equation}\label{Deps}
D^\eps:=\cup_{k=k_i} D_k^\eps \cup_i \{|x \cdot \nu_i| \le C\eps^{1/2} \},
\end{equation}
with $D_k ^\eps$ and $\nu_i$ as above. The set $D^\eps$ lies in a $C \eps^{1/2}$ neighborhood of a finite number of rays. As a consequence we have the analogue of Lemma \ref{l300} in our setting.

\begin{lem}\label{l51}
Assume that $ {\bf u} \in \mathcal S (1,{\bf p}_*, \eps )$ is $\eps$-approximated by ${\bf v}$ in $B_1$.
 
 Then in $B_{3/4}$ we have $\Gamma_k \subset D^\eps$ for all $k$, and
 \begin{equation}\label{e52}
|\triangle (u_k - v_k| \le \delta \eps \quad \mbox{in} \quad B_{3/4} \setminus D^\eps. 
\end{equation}
\end{lem}

In the next lemma we establish a H\"older modulus of continuity for the rescaled differences $(u_k-v_k)/\eps$.

\begin{lem}\label{l52}
Assume that $ {\bf u} \in \mathcal S (1,{\bf p}_*, \eps )$ is $\eps$-approximated by ${\bf v}$ in $B_1$. 

Fix $z \in B_{3/4} \setminus B_{1/4},$ and $r \in [C \eps^{1/2}, c]$. We have 
$$ w_k - 2 \eps r^\alpha \le u_k \le w_k + 2 \eps r ^\alpha \quad \quad \mbox{in $B_r(z)$, for some $\alpha>0$,}$$
with ${\bf w}$ an admissible function in $B_r(z)$ obtained from ${\bf v}$ by appropriate translating constants $\zeta_k$ (depending on $r$ and $z$),
$$w_k:=  v_k + \zeta_k, \quad \quad \quad w_k \ge w_{k+1} \quad \forall k.$$ 
Moreover, if $B_r(z)$ intersects $\{ x \cdot \nu_i\}=0$ then the constants $\zeta_k$ are all equal when $k$ belongs to the $i$-th group  $k \in \{ k_{i-1} +1,.., k_{i} \}$.
\end{lem}

We postpone the proof of Lemma \ref{l52} to the end of this section. As a consequence we obtain the following version of Corollary \ref{C1} in our setting. The difference is that, in the limit, the rescaled errors must agree along the direction of rotation for each of the connected groups of the limiting cone ${\bf p}$.

\begin{cor}\label{C1*}
If $ {\bf u}_m \in \mathcal S (1,{\bf p}_*, \eps_m ),$ are $\eps_m$-approximated by ${\bf v}_m$, for a sequence of $\eps_m \to 0$, then, up to a subsequence, ${\bf v}_m \to {\bf p} \in \mathcal P({\bf p}_*)$ and each of the rescaled error functions
$$\eps_m^{-1} \left(u_{m,j} - v_{m,j}\right)$$ converges uniformly on compact sets of $B_{1/2} \setminus\{0\}$ to a continuous limit $w_j$ that satisfies 
$$\|w_j\|_{L^\infty} \le 1, \quad w_j =w_l  \quad \mbox{on $\{x\cdot \nu_i=0\}$, whenever } \quad j,l \in \{k_{i-1}+1,..,k_i \},$$
where $\nu_i$ is the direction of rotation for ${\bf p}^i$.
\end{cor}

Another consequence of Lemma \ref{l52} is that the corresponding version of Lemma \ref{L41} holds in the degenerate setting.

\begin{lem}\label{L41*}
Assume that $ {\bf u} \in \mathcal S (1,{\bf p}_*, \eps ).$ Then
$$W({\bf u},1/2) \le W({\bf p}) +  C \eps^{3/2}.$$
\end{lem}

\begin{proof}
First we remark that $W({\bf p})$ is the same for all ${\bf p} \in \mathcal P({\bf p}_*)$. 

The quantity
$$ J(w)=\int_{B_1} \frac 12 \omega(|\nabla w|^2 + f w) \, dx - \int_{\partial B_1} \omega \, w^2 \, d \sigma,$$
remains invariant if we replace $w$ by $w + q$ with $q$ a homogenous of degree 2 harmonic polynomial (here $f$ and $\omega$ are constants). This 
follows easily after applying the mean value property for $q$ and then by integration by parts.

From \eqref{degp}, we see that each of the connected groups ${\bf p}^i + q_i {\bf 1}$ that form ${\bf p}$, 
is obtained from $i$th connected group of the trivial extension of ${\bf p}^*$ to 2D, 
after a rotation and the addition of a homogenous of degree 2 harmonic polynomial. The remark above implies $ W({\bf p})=W({\bf p}^*)$. 

The proof follows from Lemma \ref{L41} since the inequalities \eqref{4100},\eqref{4101}  i.e.
\begin{equation}\label{4100*}
W({\bf u},1/2) \le W({\bf v}) + C \eps^2,
\end{equation}
and
\begin{equation}\label{4101*}
W({\bf v}) \le W({\bf p}) + C \eps^{3/2}.
\end{equation}
continue to hold, where ${\bf v}$ is the $\eps$-approximation of ${\bf u}$ given in Definition \ref{S_sec}.

Indeed, for \eqref{4100*} we only used that $\eps^{-1}|\nabla (u_k-v_k)|$ is integrable on $\partial B_{1/2}$ which, as in Section 4, is a consequence of 
Lemmas \ref{l51} and \ref{l52}. 

The second inequality can be reduced to the one from Section 4 for each of the connected groups. 
Recall that the $i$th connected groups of ${\bf v}$, and ${\bf p}$ are given by
$${\bf p}^i(\cdot, {\bf b}_i) + q_i{\bf 1} \quad \quad \mbox{and} \quad {\bf p}^i + q_i {\bf 1}.$$

We claim that
\begin{equation}\label{wip}
W({\bf v})-W({\bf p}) = \sum_i  \, W^i({\bf p}^i(\cdot, {\bf b}_i) )-W^i({\bf p}^i) \le C \eps^{3/2},
\end{equation}
where $W^i$ denotes the Weiss energy corresponding to the $i$-th connected group
$$W^i({\bf w}^i):= \sum_{k_{i-1}<k\le k_i } \left(\int_{B_1}  \omega_k(\frac{1}{2}|\nabla w_k|^2+f^i_kw_k) \, dx - \int_{\partial B_1} \omega_k w_k^2 d \sigma \right),$$
with $f_k^i:=f_k - \triangle q_i$.
The equality in \eqref{wip} follows easily from the identity
$$ J(w+q) - J(v+q) = J(w) - J(v) - \int_{B_1} \omega \, (\triangle q)(w-v) dx, $$
which holds for any homogenous quadratic polynomial $q$.

\end{proof}

We are ready to prove the corresponding version of Proposition \ref{PM} for degenerate cones ${\bf p}_*$.

\begin{prop}\label{PM*}
Assume that $ {\bf u} \in \mathcal S (1,{\bf p}_*, \eps ),$ with $\eps \le \eps_0$. Then either
$$ {\bf u} \in \mathcal S (\rho,{\bf p}_*, \frac \eps 2),$$
or
$$ {\bf u} \in \mathcal S (\rho,{\bf p}_*, C \eps ), \quad \mbox{and} \quad W({\bf u}, \rho) \le W({\bf u},1) - c  \, \eps^2.$$
\end{prop}

\begin{proof}
As before we prove the statement by compactness. 

We fix $\rho=1/4$, $C=\rho^{-2}$, and assume that there exists a sequence of ${\bf u}_m$, ${\bf v}_m$, $\eps_m \to 0$ for which the conclusion 
does not hold with $c_m = 1/m \to 0$. 

By Corollary \ref{C1*} we may extract a subsequence $${\bf v}_m \to {\bf p} \in {\mathcal P}({\bf p}_*),$$ and rescaled errors
$${\bf w}_m:=\eps_m^{-1} \left({\bf u}_{m} - {\bf v}_m\right)$$ 
which converge uniformly of compact sets of $B_{1/2} \setminus \{0\}$ to a limit function ${\bf w}$. 

Denote by $\nu_i$ the direction of rotation for the $i$th connected cone ${\bf p}^i$ of ${\bf p}$, and by $\Gamma_{k_i}$ the coincidence set $\{ p_{k}
=p_{k+1}\}$ for $k=k_i$, which by Lemma \ref{2ray} consists of at most 2 rays that form an obtuse angle. 
The sets $D^\eps$ defined in \eqref{Deps} converge in the Haussdorff distance to the collection of rays
$$D^0:= \cup \Gamma_{k_i} \cup_i \{x \cdot \nu_i =0\},$$
and the convergence of ${\bf w}_m$ to ${\bf w}$ is in $C_{loc}^1 (B_{1/2} \setminus D^0)$. As in the proof of Proposition \ref{PM}, the inequality
$$W({\bf u}_m,1) - W({\bf u}_m, \rho) \le  c_m  \, \eps_m^2,$$
implies that the limit $\bf w$ is homogenous of degree $2$ in $(B_{1/2} - B_{\rho}) \setminus D^0$, hence in $B_{1/2} \setminus B_\rho$ by continuity. 

{\it Claim:} If $k$ belongs to the $i$-th connected group $J_i:=\{ k_{i-1}+1,.., k_i\}$ then
$$ w_k=w_l \quad \mbox{on} \quad \{x \cdot \nu_i=0\}, \quad \forall k,j \in J_i,$$
\begin{equation}\label{wham}
\triangle \, \, w_{J_i} =0, \quad \quad w_{J_i}:=\sum_{k \in J_i} \frac{\omega_k}{\sum_{J_i} \omega_j} \, w_k,
\end{equation}
and on each half space determined by the line $x \cdot \nu_i=0$ 
\begin{equation}\label{wsub}
\triangle w_j =-2e_j({\bf b}^i) \chi_{\{x \cdot \nu_i^\perp  >0 \}} -  2e_j(-{\bf b}^i) \chi_{\{x \cdot \nu_i^\perp  <0 \}},
\end{equation}
where 
${\bf b}^i \in B({\bf p}^i)$ is the limit of  
$$ {\bf b}^i:= \lim_{m \to \infty} \eps_m^{-\frac 12}{\bf b}^i_m, \quad \quad |{\bf b}^i| \le \delta.$$

{\it Proof of Claim:} 
Notice that $$\triangle u_{m,L} \le f_L = \triangle v_{m,L}, \quad \quad L:=\{j \le k_j\},$$
which implies that 
$$\triangle w_L \le 0.$$
On the other hand outside any small neighborhood of $\Gamma_{k_i}$, $p_{k_i}> p_{k_i+1}$ which implies the same inequality for the membranes of ${\bf u}_m$. This means that the inequality above is an equality, which gives 
$$\triangle w_L =0 \quad \mbox{outside $\Gamma_{k_i}$.}$$
Since $w_L$ is homogenous of degree two and $\Gamma_{k_i}$ consists of at most 2 rays that form an angle different than $\pi/2$ we conclude that $w_L$ must be a harmonic quadratic polynomial. This implies \eqref{wham}.

The equality \eqref{wsub} follows in $B_{1/2} \setminus D^0$ by  \eqref{bfvi}. In fact it can only fail on the rays $\Gamma_{k_{i-1}} \cup \Gamma_{k_i}$ along which the $i$th connected group can interact with the $i-1$ respectively $i+1$ groups. Indeed, in a compact set outside these rays the graphs of $u_k$ with $k \in J_i$ are disconnected from the ones with $k \notin J_i$, and we are in the situation of Section 4. More precisely, we only need to check \eqref{wsub} for those indices $j\in J_i$ and near the rays for which the membrane $p_j$ is either tangent to $p_{k_i+1}$ or $p_{k_{i-1}}$.

It remains to show that if the membrane $p_j$ is tangent to $p_{k_i+1}$, then $\triangle w_j$ carries no singular part on $\Gamma_{k_i}$ whenever $\Gamma_{k_i}$ is not included in $ x \cdot \nu_i=0$.  
Pick such a ray $\ell \in \Gamma_{k_i} \setminus \{x \cdot \nu_i=0 \}$ and let $J_i' \subset J_i$ denote those indices $j$ in the $i$th group for which $p_j=p_{k_i}$ along $\ell$. Since $\ell$ is away from the line $x \cdot \nu_i =0$ we conclude that $p_j=p_{k_i}$ in a neighborhood of $\ell$. Using that ${\bf v}_m$, ${\bf u}_m$ are small perturbations of ${\bf p}$ we find that in an open neighborhood $\mathcal U$ of  $\ell  \cap (B_1 \setminus B_\rho)$, 
$$v_j=v_{k_i}, \quad  u_j=u_{k_i} \quad \mbox{ if $ j\in J_i'$}.$$
In particular in this neighborhood $w_j=w_{k_i}$ if $j \in J_i'$, hence
$$\triangle w_j = \triangle w_{J_i'}  \quad \mbox{in $\mathcal U$}. $$
If $J_i'=J_i$ then $\triangle w_j=0$ by \eqref{wham} which shows that $\triangle w_j$ has no singular part on $\ell$. If $J_i' \ne J_i$ then 
there is strict separation in $\mathcal U$ between the membranes $p_j$ with $j \in J_i'$ and $j \in L \setminus J_i'$. This separation holds also for the membranes of ${\bf u}_m$, and ${\bf v}_m$ hence 
$$\triangle u_{m,L \setminus J_i'}=f_{L\setminus J_i' },$$
and since ${\bf v}_m$ is an approximate solution with $\delta \eps_m$ error we find that
$$ |\triangle w_{L \setminus J_i'}| \le \delta \quad \mbox{in $\mathcal U$.}$$
Using that $w_L$ is harmonic we find $|\triangle w_{J_i'} | \le C \delta$. This shows that $\triangle w_j$ has no singular part on $\ell$ if $j \in J_i'$, and the claim is proved.

\qed

Now we can argue as in the end of the proof of Proposition \ref{PM}. The claim implies that ${\bf e}({\bf b}^i)={\bf e}(- {\bf b}^i)$, hence, by Lemma \ref{eb},
 $${\bf b}^i=s_i \tau^i \quad \mbox{ for some $s_i \in [-C \delta,C \delta]$},$$ 
 and $\tau^i$ as in Definition \ref{tau}. Moreover, for $j \in J_i$,
$$w_j:= \bar q_i + \gamma_j (x \cdot \nu)^2 + \left (t_j^+  \chi_{\{x \cdot \nu >0\}} + t_j^- \chi_{\{x\cdot \nu_i<0\}}\right) (x \cdot \nu_i) (x \cdot \nu_i^\perp), $$
with $\gamma_j=-e_j({\bf b}^i)$, $\bar q_i= w_{J_i}$ a harmonic quadratic polynomial, and the components $t_j^\pm$ form a vector ${\bf t}^i \in B({\bf p}^i)$.
Since $|\gamma| \le C \delta^2$, we infer that
$$|{\bf u}^i_m - [{\bf p}^i(x,{\bf b}^i_m + \eps_m {\bf t}^i) + (q_i +\eps_m \bar q_i) {\bf 1}]| \le C \delta^2 \eps_m \, \rho^2 \quad \mbox{in $B_{2 \rho} \setminus B_\rho$.}$$
As in \eqref{change}, we can rotate the axis $\nu_i$ of ${\bf p}^i$ by an angle $\sim \eps_m^{1/2}$ and rewrite
$${\bf p}^i(x,{\bf b}^i_m + \eps_m {\bf t}^i)={\bf p}^i (\tilde x,\eps_m^{1/2} {\bf d}^i_m) + O (\delta \eps_m |x|^2), \quad \quad {\bf d}^i_m \to 0,$$
with $\tilde x$ representing the coordinates in the rotated system of coordinates. Thus
\begin{equation}\label{bfui}
|{\bf u}^i_m - \tilde {\bf v}^i_m| \le C \delta \eps_m \, \rho^2 \quad \mbox{in $B_{2 \rho} \setminus B_\rho$,}
\end{equation}
with
$$ \tilde {\bf v}^i_m:= {\bf p}^i (\tilde x,\eps_m^{1/2} {\bf d}^i_m) + q_i + \eps_m \bar q_i.$$
We don't know yet that the family $\tilde {\bf v}$ is admissible since the inequality $\tilde v_{m,k} \ge \tilde v_{m,k+1}$ might fail slightly when $k=k_i$ 
near $\Gamma_{k_i}$. By \eqref{bfui}, this inequality can fail by at most $C \delta \eps_m |x|^2$. We can modify each group of $\tilde {\bf v}_m$ by a harmonic quadratic polynomial of size $\delta \eps_m$, and construct an admissible approximate solution $\bar {\bf v}_m$. Indeed, assume that 
$\bar{\bf v}_m^1$,..,$\bar{\bf v}_m^{i-1}$ were constructed. Then we can add $C_i \delta \eps_m h_i(x)$ to all membranes of $\tilde {\bf v}_m^{i}$ with $h_i$ a harmonic quadratic polynomial which is negative on $\Gamma_{k_i}\setminus \{0\}$, which exists in view of Lemma \ref{2ray}. We can choose $C_i$ sufficiently large to guarantee that $\bar {\bf v}_m^i$ lies below $\bar {\bf v}_m^{i-1}$. After constructing $\bar {\bf v}_m$, we can subtract its average (a harmonic polynomial) from all of its components, so that $\sum \omega_k \bar v_{m,k}=0$. In conclusion, \eqref{bfui} implies that
$$|{\bf u}^i_m - \bar {\bf v}^i_m| \le C' \delta \eps_m \, \rho^2 \quad \mbox{in $B_{2 \rho} \setminus B_\rho$,}$$
with $\bar {\bf v}_m$ satisfying the admissible conditions \eqref{bfv} with $\eps_m$ replaced by $\eps_m/2$.

Finally, since $\bar {\bf v}_m$ solves the system with error $\delta \eps_m$, it follows by maximum principle that the inequality above can be extended to $B_\rho$ after relabeling the constant  $C'$. Thus
$$|{\bf u}^i_m - \bar {\bf v}^i_m| \le C'' \delta \eps_m \, \rho^2 \le \frac{\eps_m}{2} \rho^2 \quad \mbox{in $B_{\rho}$,}$$
provided $\delta$ is chosen small. We obtain ${\bf u}_m \in \mathcal S ({\bf p}_*,\rho, \eps_m/2)$ and reached a contradiction.
\end{proof}

The remaining of the section is devoted to the proof of Lemma \ref{l52} which relies on a version of the Harnack inequality for 1D membranes. 
 
 \begin{lem}\label{Han1d}
 Assume that ${\bf u} \ge {\bf v}$ are 1D solutions to the $N$ membrane problem in $[-1,1]$ and $$\mbox{$u_k(0) \le v_k(0) + \sigma$, for some $k$ and $\sigma \ge 0$.}$$ 
 Then $$u_k \le v_k + C \sigma \quad \mbox{ in $[-1,1]$}$$ 
 for some $C$ depending only on $N$, and the weights $\omega_i$.
 \end{lem}
 
 \begin{proof}
 We prove the statement by induction on the cardinality of the complement of the set of indices $I$ defined as
 $$ I:=\{j| \quad u_j(0) \le v_j(0) + a \}.$$
 Precisely we show that there exists a constant $C(|I|)$ depending only on the cardinality $|I|$ of the set $I$, such that in $[-1,1]$
 $$u_j \le v_j + C(|I|) \sigma, \quad  \quad \forall j \in I.$$
 If $|I|=N$, then $I=\{1,..,N\}$. We have $v_I(0)+ \sigma \ge u_I(0) \ge v_I(0)$ and since $u_I - v_I \ge 0$ is harmonic in $[-1,1]$ we conclude that $u_I \le v_I + 2 \sigma$ which gives the desired conclusion.
 
 Assume that $|I| <N$, and denote $I=\{j_0, .., j_0 + m\}$. Let $(a, b)$ be the largest interval containing $0$ on which the inequalities 
 $$\mbox{$u_{j_0-1} > u_{j_0}$ 
 and $v_{j_0+m} > v_{j_0+m+1}$ hold.}$$
 Notice that the origin is interior to this interval, since otherwise either $j_0-1$ or $j_0+m+1$ would belong to $I$ as well. 
 
 Assume that $|a| \le |b|$ and say that $u_{j_0-1}(a)=u_{j_0}(a)$. 
 
 In the interval $(a,b)$ the same argument as above applies. 
 Indeed, in this interval the membranes $u_j$, and respectively $v_j$, with $j \in I$, can be perturbed upwards, and respectively downwards. We find
 $ \triangle u_I \le f_I \le \triangle v_I $ hence $u_I - v_I \ge 0$ is a concave function in $(a,b)$. We conclude that 
 \begin{equation}\label{ulevj}
 u_j \le v_j + C_1 \sigma \quad \mbox{in $[a,|a|]$} \quad \quad \forall j \in I.
 \end{equation}
 In particular at $x=a$ we have $$u_k=u_{k+1} \le v_{k+1} + C_1 \sigma \le v_k + C_1 \sigma, \quad k=j_0-1.$$ 
 We can apply the induction hypothesis on the largest interval $L_a$ centered at $a$ which is included in $[-1,1]$ with $\tilde \sigma = C_1 \sigma$, and then the corresponding set of indices $\tilde I$ contains $I$ and $j_0-1$. We find that
 \begin{equation}\label{ulevj2}
 u_j \le v_j + C_2 \sigma \quad \mbox{in $L_a$} \quad \quad \forall j \in I \cup \{j_0-1\}.
 \end{equation}
 If $L_a$ contains the origin, then we can apply one more time the induction hypothesis at the origin and obtain the desired conclusion in the whole 
 interval $[-1,1]$. Otherwise, the inequality \eqref{ulevj} is valid in $[a,b]$ after relabeling $C_1$ if necessary. We can argue as above at the other end point $b$ and obtain a similar inequality as \eqref{ulevj2} in the largest interval $L_b \subset [-1,1]$ centered at $b$. Since $[-1,1]$ is covered by $L_a$, $[a,b]$ and $L_b$ we obtain the inductive conclusion for $I$.    

 \end{proof}

 We introduce the notion of $\sigma$-connectedness in $B_r \subset \R^n$ for membranes whose collection of $\sigma$-neighborhood of their graphs form a connected set. 

 \begin{defi}\label{scon}
 We say that the membranes $v_{j}$ and $v_{j+m}$ are $\sigma$-connected in $B_r$ if we can find points $x_i \in B_r$ with $j+1 \le i \le j+m$ such that $v_{i-1}(x_i) \le v_{i}(x_i) + \sigma$.
 \end{defi}

 \begin{rem}
After relabeling the constant $C$, the conclusion of Lemma \ref{Han1d} holds for all indices $j \le k$ for which $u_j$ is $\sigma$-connected to $u_k$ in the half-interval $$I:=[- \frac 12, \frac 12],$$ or $j \ge k$ for which $v_j$ is $\sigma$-connected to $v_k$ in $I$. 
 \end{rem}
 
 An equivalent statement is the following.

 \begin{cor}\label{c51}
 Assume that ${\bf u} \ge {\bf v}$ are 1D solutions to the $N$ membrane problem in $[-1,1]$ and $$\mbox{$u_k(1) \ge v_k(1) + \sigma$, for some $k$ and $\sigma \ge 0$.}$$ Then
 $$ u_j \ge v_j + c \sigma \quad \mbox{in} \quad I,$$
 for all $j \le k$ for which $v_j$ is $c\sigma$-connected to $v_k$ in $I$, and all $j \ge k$ for which $u_k$ is $c\sigma$-connected to $u_k$ in $I$. Here $c=C^{-1}$ depends only on $N$ and $\omega_i$.
 \end{cor}
 
 We now consider the the case when ${\bf u}$ is defined in the cylindrical domain $$\mathcal R:= B_{C_n}' \times [-1,1]\subset \R^n,$$ with $C_n$ a large constant that depends only on $n$ and ${\bf v}$ is one-dimensional and
  \begin{equation}\label{odap}
 \mbox{${\bf v}$ solves the Euler-Lagrange equation in $[-1,1]$ with a $c_0 \sigma$ - error,}
 \end{equation}
 for some $c_0$ sufficiently small. 
  
  \begin{lem}\label{Han1d2}
 Assume that ${\bf u}$ is a solution in $\mathcal R$ and ${\bf v}$ satisfies \eqref{odap} and
 $${\bf u}(x',x_n) \ge {\bf v}(x_n) \quad \mbox{in} \quad \mathcal R,$$  
 and $$ u_k(x',l) \ge v_k(l) + \sigma \quad \mbox{for some $l \in [-1,1]$}.$$ 
 for some $\sigma \le \sigma_0$ universal. Then 
$$ u_j \ge v_j + c_0 \sigma \quad \mbox{in} \quad \frac 12 \, \mathcal R,$$
 for all $j \in J_k$ which consists of the indices $j$ such that
 
 a) either $j \le k$ and $v_j$ is $c_0\sigma$-connected to $v_k$ in $I$, 
 
 b) or $j >k$ and the coincidence sets $\{v_k=v_{k+1}\}$, $\{v_{k+1}=v_{k+2}\}$,..,$\{v_{j-1}=v_j\}$ have length more than $1/10$ in $I$. 
 \end{lem}
 
 \begin{rem}\label{r52}
 Notice that the collection of functions $v_j$ when $j \notin J_k$ and $v_j + c_0 \sigma$ when $j \in J_k$, which bounds $u_j$ by below, is admissible in $\frac 12 \mathcal R$. 
 \end{rem}
 \begin{proof} We assume first that $l=1$ and then explain how to deduce the more general statement from this case.
 
 Let ${\bf w}$ be the 1D solution in $[-1,1]$ with the boundary data given by ${\bf v}$. 
 
 We compare ${\bf w}$ with ${\bf v} \pm c_0 \sigma (|x|^2-1) \, {\bf 1} $ in $
 [-1,1]$ and find
 \begin{equation}\label{barw-v}
 |w_j - v_j| \le c_0 \sigma \quad \quad \forall j.
 \end{equation} 
 In particular $w_j$ are $3 c_0 \sigma$-connected in $I$ if $j \le k$ and $j \in J_k$.
  
 Let $\bar {\bf w}$ be the 1D solution with boundary data ${\bf w}$ at $-1$ and 
 $$\mbox{$\bar {w}_j (1)= w_j(1)$ if $j >k$ and $\bar {w}_j (1)= \max \{ w_j(1), w_k(1) + \sigma\} $ if $j \le k$. }$$
 Clearly
 $$|\bar w_j - w_j| \le \sigma \quad \forall j, \quad \mbox{in} \quad [-1,1],$$
 which together with \eqref{barw-v} and $\sigma \le \sigma_0$ implies that the $\bar w_j$ are $0$-connected in $I$ if $j > k$ and $j \in J_k$.
 
 By Corollary \ref{c51} applied to $\bar {\bf w}$, ${\bf w}$, we can find $c_1=c_1(N,\omega_i)$ such that
 \begin{equation}\label{5002}
 \bar w_j \ge w_j + 4c_1 \sigma \ge v_j + 3 c_1 \sigma \quad \mbox{in} \quad I,\quad \quad  \forall j \in J_k,
 \end{equation}
provided we choose $c_0 \le c_1$.

 Next we compare ${\bf u}$ with the subsolution $$\bar {\bf w} + c_1 \sigma (x_n^2 - 4 C_n^{-2} |x'|^2 -1) \,  {\bf 1},$$
 in $\mathcal R$ and obtain
 $$ u_j \ge \bar w_j - 2 c_1 \sigma \quad \mbox{in} \quad  \frac 12 \mathcal R , \quad \forall j,$$
 which, by \eqref{5002}, gives the conclusion $u_j \ge v_j + c_1 \sigma$ for all $j \in J_k$.
 
 It suffices to check the claim on $\partial \mathcal R$. On $\partial \mathcal R \setminus \{x_n=1\}$ the test function is below ${\bf v}$ and therefore below ${\bf u}$. This inequality holds also on $\partial \mathcal R  \cap \{ x_n =1\}$ by hypothesis. This completes the case $l=1$.

Next we discuss the case when $l$ is arbitrary. The same proof applies if $|l| \ge 3/4$. In the case when, say $l \in [0, 3/4)$, then the arguments above show that an inequality of the form $$u_k (x', -3/4) \ge v_k(-3/4) + c_1' \sigma \quad \mbox{ if} \quad |x'| \le \frac 34 C_n,$$
holds for the index $j=k$ at $-3/4$. 
Again we may repeat that  proof above with $\tilde l = -3/4$ and $\tilde \sigma = c_1' \sigma$, and obtain the conclusion by choosing $c_0$ much smaller if necessary.  

 \end{proof}
 
 We provide a version of Lemma \ref{Han1d2} when ${\bf v}$ is a homogenous of degree 2 approximate solution in a rectangular domain in polar coordinates $\mathcal R_\tau \subset \R^2$ defined as
 \begin{equation}\label{rtau}
 \mathcal R_\tau:=\{(r, \theta)| \quad |\theta| \le \tau, \quad |r-1| \le C \tau \}, \quad \quad \mbox{with} \quad \tau < \tau_0.
 \end{equation}
 
 \begin{lem}\label{Han1d3}
 Assume that ${\bf u}$ is a solution to the $N$-membrane problem in $\mathcal R_\tau$, and ${\bf v}$ is a $C^{1,1}$ homogenous of degree 2 function which solves the Euler-Lagrange equation in $\mathcal R _\tau$ with $c_0 \sigma \tau^{-2}$ error. 
 
 If ${\bf u} \ge {\bf v}$ in $\mathcal R_\tau$, and $u_k \ge v_k + \sigma|x|^2$ on a ray $\mathcal R_\tau \cap \{\theta= l\}$ then
 $$ u_j \ge v_j + c_0 \sigma |x|^2 \quad \mbox{in} \quad \mathcal R_{\tau/2},$$
 for all $j \in J_k$ for which either $j \le k$ and $v _j$ is $c_0\sigma$-connected to $v_k$, or $j > k$ and the coincidence sets $\{v_k=v_{k+1}\}$, $\{v_{k+1}=v_{k+2}\}$,..,$\{v_{j-1}=v_j\}$ have length more than $\tau/10$ in the interval $ \theta \in [-\tau/2,\tau/2]$.
 \end{lem}
 
 The proof of Lemma \ref{Han1d3} follows as the one of Lemma \ref{Han1d2} after we establish a version of the 1D lemma, Lemma \ref{Han1d}, on the unit circle. We omit the details but point out some of the changes in this setting.
 
  We consider functions ${\bf v}$ on small intervals $[-\tau,\tau]$ on the unit circle which solve the $N$-membrane problem for the operator $-
  \partial_{\theta 
  \theta} - 4$ which is positive definite if $\tau < \pi/4$. Then the homogenous 2 extension of ${\bf v}$ solves the $N$-membrane problem in the 
  corresponding sector in $\R^2$. The energy corresponding to the new operator has the form 
  $$ \int_{-\rho}^\rho \sum \omega_k\left(\frac 12 |v_k'|^2 - 2 v_k^2 + f_k v_k\right) \, d \theta,$$
  end the existence of solutions follows in the same way as before. The proof of Lemma \ref{Han1d} is identical since the following Harnack inequality continues to hold:
$$ \partial_{\theta \theta} w + 4 w \le 0, \quad \mbox{and} \quad w \ge 0 \quad  \Longrightarrow \quad w \le C w(0) \quad \mbox{ in $[-\tau,\tau]$}.$$
We are ready to prove Lemma \ref{l52} by comparing ${\bf u}$ with appropriate translations of ${\bf v}$ that are homogenous of degree 2, and make use of Lemma \ref{Han1d3} above.

\begin{proof}[Proof of Lemma \ref{l52}]
Assume for simplicity that $z = 1/2 e_1$, and choose $\rho$ universal such that (see \eqref{rtau}) $\mathcal R_{4 \rho r} \subset B_r(e_1).$

We prove by induction on $m \ge 0$ that in $B_{r}(z)$ with $r = \bar c \rho^{m}$, for some $\bar c$ small to be specified later, as long as $r \ge C \eps^{1/2}$ we have
\begin{equation}\label{5000}
v_k + \zeta_{k,m}^- |x|^2 \le u_k \le  v_k + \zeta_{k,m}^+ |x|^2,
\end{equation}
$$\zeta_{k,m}^\pm = \zeta_{k,m} \pm \eps_m, \quad \quad \eps_m:=8 (1-c)^m \eps,$$
for some $c>0$ small universal, and constants $\zeta_{k,m}$ for which $v_k + \zeta_{k,m} |x|^2$ is admissible. 

Moreover, the constants $\zeta_{k,m}$ are all equal when $k$ belongs to the $i$-th group  $k \in \{ k_{i-1} +1,.., k_{i} \}$ and 
\begin{equation}\label{vkevk}
\mbox{the line $\{ x \cdot \nu_i\}=0$ intersects $B_r(z)$.}
\end{equation}

Notice that our hypothesis $|{\bf u}-{\bf v}| \le \eps$ implies that $\zeta_{k,m} \in [-16 \eps, 16 \eps]$.

When $m=0$ we can take $\zeta_{k,0}= 0$ by hypothesis.

Assume the induction hypothesis holds for $r=r_m$. We want to show that that \eqref{5000} holds in $B_{\rho r}(z)$ for some constants $\xi_{k}^\pm$ with 
$$\zeta_{k,m}^-\le \xi_k^- \le \xi_k^+ \le \zeta_{k,m}^+, \quad \quad \xi_k^+ - \xi_k^- \le (1-c)(\zeta_{k,m}^+-\zeta_{k,m}^-),$$
and $v_k + \xi_k^\pm |x|^2$ are admissible, and with $\xi^\pm_{k+1}=\xi_k^\pm$ whenever the 
condition \eqref{vkevk} holds for $B_{\rho r}$. 
Then we define $\zeta_{k,m+1}$ as the averages of $\xi_k^\pm$ and the conclusion follows for $m+1$.

We pick a unit 
direction $\bar \nu$ close to the direction $e_1$ of $z$
$$ |\bar \nu- e_1| \le \rho r,$$ 
such that a $c \, r$ neighborhood of the ray of direction $\bar \nu$ does not intersect the set $D^\eps$ (defined in \eqref{Deps}) in 
$B_r (z)$. This is possible since $r \ge C \eps^{1/2}$. Assume that at $\frac 12 \bar \nu$, $u_k$ is closer to the upper bound in \eqref{5000} i.e.
\begin{equation}\label{5001}
u_k (\frac 12 \bar \nu) \ge (v_k  + (\zeta^-_{k,m} + \eps_m)|x|^2)(\frac 12 \bar \nu).
\end{equation}
By Lemma \ref{l51}, outside $D^\eps$
$$\left|\triangle \left(u_k- (v_k  + \zeta_{k,m}^- |x|^2) \right)\right| \le \delta \eps + 2 |\zeta_{k,m}^-| \le 40 \eps \le  \bar c \eps_m r^{-2}.$$  
By Harnack inequality applied to the difference $$u_k- \left(v_k  + \zeta_{k,m}^- |x|^2 \right) \ge 0$$ we find that \eqref{5001} can be extended to
$$ u_k \ge v_k  + \zeta_{k,m}^- |x|^2 + c' \eps_m \ge v_k  + (\zeta_{k,m}^-+ c'\eps_m) |x|^2,$$
for some $c'$ universal on the whole ray 
 $$ B_{r/2}(z) \cap \{t \bar \nu| \quad t \ge 0\},$$
 provided that $\bar c$ is sufficiently small.
Now we can apply Lemma \ref{Han1d3} to $u_k(x)$ (in fact the quadratic rescalings $4 u_k(x/2) $) and  $v_k  + \zeta_{k,m}^- |x|^2$ in $\mathcal R_{8 \rho r}$ with $\sigma:=c' \eps_m$, 
since the error for the approximate solutions is bounded by
$$40 \eps \le \bar c \eps_m r^{-2} \le c_0 \sigma (8 \rho \, r)^{-2},$$ 
and obtain
$$ u_j \ge v_j  + (\zeta_{j,m}^- +c''\eps_m) |x|^2,$$
in $B_{\rho r}(z)$ for all $j \in J_k$, for some $c''$ small, universal. As in Remark \ref{r52}, the righthand sides correspond to an admissible family in $B_{\rho r}(z)$. Moreover, they change by the same amount on a set of indices $j$ that belong to an $i$-th group $\{k_{i-1} +1,.., k_{i} \}$ for which $\{x \cdot \nu_i\}$ intersects $B_{\rho r} (z)$, since in this case the coincidence sets $\{v_{j-1}=v_{j}\}$ cover more than $1/10$ of the interval $\theta \in [-4 \rho r, 4 \rho r]$  on the unit circle $\partial B_1 $. This means that we can choose $\xi_k^\pm$ accordingly in $B_{\rho r}$ and the lemma is proved.

\end{proof}

 %%%%%%%%%%%%%%%%%%%%%%%%%%%%%%%
 %%%%%%%%%%%%%%%%%%%%%%%%%%%%%%%
 
 \section{Regular intersection points}
 
 In this section we study the regularity of the free boundaries for solutions ${\bf u}$ that stay close to the blow-up cone
 $${\bf p_0}(x):= \frac 12 (x_2^+)^2 {\bf f},$$
and prove Theorem \ref{TIntro2} which we recall.
 \begin{thm}\label{Treg}
 Assume $d=2$ and $$|{\bf u}-{\bf p}_0| \le \eps_0 \quad \mbox{in $B_1$.}$$ 
 Then each $\Gamma_i$ is a $C^{1,log}$ curve in $B_{1/2}$.
 \end{thm}
 
 We prove Theorem \ref{Treg} by induction on the number of membranes $N$. One of the technical points is that we need a lower bound for the Weiss energy, see Lemma \ref{wpb3}, which is not obvious since we no longer assume $0 \in \cap \Gamma_i$.
 
 Similar to Definition \ref{S}, we approximate solutions ${\bf u}$ by the slightly more general functions from Definition \ref{pbb} 
 $$ {\bf p}(x,{\bf b}_0,{\bf b}_1)={\bf h}(x_2,{\bf b}_0 + x_1 {\bf b}_1) , \quad {\bf b}_i \in {\bf B}({\bf p}_0).$$

 \begin{prop}\label{P5}
 Assume that a solution $\bf u$ to the problem $P_0$ satisfies
 \begin{equation}\label{6000}
 |{\bf u} - {\bf p}(\cdot, {\bf b}_0,{\bf b}_1)| \le \eps r^2 \quad \mbox{in $B_r$}, 
 \end{equation}
$$\mbox{for some} \quad {\bf b}_i \in B({\bf p}_0), \quad \mbox{with} \quad |{\bf b}_0| \le \eps^{1/2}r, \quad |{\bf b}_1| \le 2 \delta \eps^{1/2}.$$
Then
 \begin{equation}\label{6002}
 |{\bf u} - {\bf p}(\cdot, {\bf b}_0',{\bf b}_1')| \le  \frac{\eps}{2}  (\rho r)^2 \quad \mbox{in $B_{\rho r}$}, 
 \end{equation}
 with ${\bf b}'_i \in B({\bf p}_0)$ and
\begin{equation}\label{6003}
 |{\bf b}'_0-{\bf b}_0| \le C_0 \eps r, \quad |{\bf b}'_1-{\bf b}_1| \le C_0 \eps.
 \end{equation}
The constant $C_0$ depends only on the dimension $d=2$, $\rho \le \rho_0$ universal, $\delta \le \delta(\rho)$ depending on $\rho$, and $\eps \le \eps_0(\delta,\rho)$ sufficiently small.
  \end{prop}
 
 After rescaling it suffices to prove the proposition for $r=1$.
 
 First we estimate the change in ${\bf h}(x,{\bf b})$ as we vary ${\bf b}$.
 
 \begin{lem}\label{hxb}
 $$\left|{\bf h}(x,{\bf b} + {\bf d}) - ({\bf h}(x,{\bf b}) + x {\bf d} )\right| \le C  |{\bf d}|(|{\bf b}| +|{\bf d}|)$$
 \end{lem}
 
 \begin{proof}
 By the homogeneity of ${\bf h}$ we may assume that $|{\bf b}|+|{\bf d}|=1$. Then by Lemma \ref{L2.2} we know that the left hand side is constant when $x$ is outside the interval $[-C,C]$. So it suffices to prove the inequality when $|x|\le C$. Now the inequality follows from the Lipschitz continuity of ${\bf h}$ in its second variable.

 \end{proof}
 
 Next we establish in the context of Proposition \ref{P5} the estimate for the rescaled error of ${\bf u}-{\bf p}$ in terms of the distance to the $x_2$ axis, as we did in Lemma \ref{l36}. 

\begin{lem}\label{l365}Assume that $ {\bf u}$ satisfies \eqref{6000} with $r=1$. Then in $B_{1/2}$
$$ |{\bf u} - {\bf p}(\cdot, {\bf b}_0,{\bf b}_1)| \le C \, \eps (|x_2|+ \sqrt \eps)^\alpha,$$
for some $\alpha>0$ small, universal.
\end{lem}

\begin{proof}
The proof is essentially the same with the one of Lemma \ref{l36}, after replacing ${\bf p}(\cdot,{\bf b})$ by ${\bf p}(\cdot,{\bf b}_0, {\bf b}_1)$. A few comments are in order. 

First we remark that the approximate solution solves the Euler-Lagrange equations with error $C |{\bf b}_1|^2 \le \delta \eps$ as before, and is not affected by the presence of ${\bf b}_0$, see Lemma \ref{L32}. 

The comparison function ${\bf v}$ in $B_{r_k}(Z)$ is defined as before
$${\bf v}(x):={\bf p}(x_2,  {\bf b}_0 + {\bf d},{\bf b}_1) + (c_1 \eps_k  q((x-Z)/r_k) -\eps_k) {\bf 1} ,$$
with ${\bf d}$, $q$ as in \eqref{dq}. The inequality \eqref{rkZ} is then replaced by
\begin{equation}\label{rkZ2}
|{\bf p}(x, {\bf b}_0 + {\bf d}, {\bf b}_1) - {\bf p}(x,  {\bf b}_0, {\bf b}_1) -  x_2 {\bf d}| \le \frac{C}{ C_1'} \eps_k \quad \mbox{in} \quad B_{r_k}(Z),
\end{equation}
and the rest of the proof remains the same, by choosing $C_1'$ sufficiently large depending on the other constants $c_1$, $c_2$ and $\mu$. We no longer use Lemma \ref{L32} to establish \eqref{rkZ2}, but Lemma \ref{l365} above with ${\bf b}={\bf b}_0 + x_1 {\bf b}_1$. Then $|{\bf b}| \le 2 \eps^{1/2}$ and, since $|{\bf d}| \le C \eps_k r_k^{-1}$ and $r_k \ge C_1' \eps^{1/2}$, the left hand side in \eqref{rkZ2} is bounded by
$$ C \eps^{1/2} \eps_k r_k^{-1} \le \frac{C}{C_1'} \eps_k.$$
\end{proof}

\begin{rem}\label{RFB}As a consequence of Lemma \ref{l365} and of the quadratic separation of consecutive membranes from their common free boundary, we find that in $B_{1/2}$ the free boundaries $\Gamma_i({\bf u})$ of ${\bf u}$ lie in a $\eps^{\frac 12 + \frac \alpha 4}$ neighborhood of the corresponding free boundaries of the approximate solution ${\bf p}(x, {\bf b}_0,{\bf b}_1)={\bf h}(x_2,{\bf b}_0+ x_1{\bf b}_1)$. In particular $\Gamma_i({\bf u})$ lie in an $C \delta \eps^{1/2}$ neighborhood of the free boundaries $x_2=\Gamma_i({\bf b}_0)$ of the exact solution ${\bf p}(x,{\bf b}_0,0)={\bf h}(x_2,{\bf b_0})$.

Assume that the free boundaries of ${\bf h}(x_2,{\bf b_0})$ separate of order $\eps^{1/2}$, i.e. there exists an interval $[a- c_0 \eps^{1/2},a + c_0 \eps^{1/2}]$ for some $c_0$ small, which does not intersect the $\Gamma_i({\bf b}_0)$, but at least one of these points falls to the left of this interval and at least one to the right. Assume $\delta \ll c_0$ is sufficiently small. Then the free boundaries $\Gamma_i({\bf u})$ do not intersect the strip
$$ S:=\{ |x_2-a| \le \frac {c_0}{2} \eps^{1/2}\},$$ and the $N$-membrane problem decouples into several multi-membrane problems in $B_{1/2}$ involving fewer membranes.

Indeed, for each set of indices $j \in J$ for which $u_j$ agree in the strip $S$, we replace $u_j$ by $u_{J}$ to the right of the strip (we think $x_2$ is the horizontal direction). If there are $J_1$,..,$J_l$ such sets, then we obtain a multi-membrane problem involving $l$-membranes. The free boundaries of the new problem coincide with the free boundaries of ${\bf u}$ that were on the left of the strip $S$. On the other hand, for each set $J$, $u_j-u_J$ solves a multi-membrane problem which has $\Gamma_j(\bf u)$ with $j \in J$ as free boundaries, which lie to the right of the strip $S$. 
The same decoupling procedure can be performed to the approximate solution ${\bf p}(x, {\bf b}_0,{\bf b}_1)$, hence the decoupled multi-membrane problems in $B_{1/2}$ are still $\eps$-approximated by corresponding functions of the type ${\bf p}(\cdot, {\bf b}_0,{\bf b}_1)$.
\end{rem}

Also Lemma \ref{l365} implies the uniform convergence of the rescaled errors. 
\begin{cor}\label{C12}
If 
$$ |{\bf u}_m - {\bf p}(\cdot, {\bf b}^m_0,{\bf b}^m_1)| \le \eps_m \quad \mbox{in $B_1$}, \quad \mbox{with} \quad |{\bf b}^m_0| \le \eps_m^{1/2}, \quad |{\bf b}^m_1| \le 2 \delta \eps_m^{1/2},$$
for a sequence of $\eps_m \to 0$, then, up to a subsequence, each of the rescaled error functions
$$\eps_m^{-1} \left(u_{m,j} - p_j(\cdot,{\bf b}^m_0, {\bf b}^m_1 )\right)$$ converges uniformly in $B_{1/2}$ to a limit $w_j$ that satisfies 
$$\|w_j\|_{L^\infty} \le 1, \quad w_j =0  \quad \mbox{on $x_2=0$,}$$
and $$|\triangle w_j| \le \delta \quad \mbox{away from $\{x_2=0\}$.}$$

\end{cor}

\

{\it Proof of Proposition \ref{P5}}

The rescaled error functions 
$$\eps^{-1} \left(u_j - p_j(\cdot,{\bf b}_0, {\bf b}_1 )\right)$$
are well approximated in $B_{1/2}$ by continuous functions $w_j$ which vanish on $x_2 \le 0$ and satisfy $|\triangle w_j|\le \delta $ in $\{x_2>0\}$. 
Denote by ${\bf d}_0$, ${\bf d}_1 \in {\bf B}({\bf p}_0)$ as
$$ d_{0,j}^+=\partial _{x_2} w_j(0), \quad d_{0,j}^-=0, \quad \quad  d_{1,j}^+=\partial _{x_1x_2} w_j(0), \quad \quad d_{1,j}^-=0.$$
Then $|{\bf d}_i| \le C_0$, and
$$ |{\bf w} - x_2 ({\bf d}_0 + x_1{\bf d}_1)| \le C_0 (\rho^3 + \delta) \quad \mbox{in $B_\rho$},$$
for a constant $C_0$ that depends only on the dimension $d=2$. If $\rho \le \rho_0$ universal, and $\delta \le \delta(\rho)$ depending on $\rho$, then the right hand side is less than $\rho^2/4$. 

By Lemma \ref{hxb}
$${\bf p}(x, {\bf b}_0 + \eps {\bf d}_0, {\bf b}_1 + \eps {\bf d}_1) -{\bf p}(x, {\bf b}_0, {\bf b}_1) =  \eps x_2 ({\bf d}_0 + x_1{\bf d}_1) + O(\eps^{3/2}),$$
and we obtain the desired result by choosing ${\bf b}_0'={\bf b}_0 + \eps {\bf d}_0$, ${\bf b}_1'={\bf b}_1 + \eps {\bf d}_1$.

\qed

\begin{rem}\label{dyc} Assume that in $B_1$ we satisfy \eqref{6000} and in addition ${\bf b}_0=0$. We have the following dichotomy depending on the size of ${\bf d}_0$ in the proof above.

a) If 
\begin{equation}\label{csub1}
|{\bf d}_0| \le c (\rho_0)=:c_1
\end{equation} then we may choose ${\bf b}_0'=0$ and satisfy the conclusion
 $$ |{\bf u} - {\bf p}(\cdot, 0,{\bf b}_1')| \le  \frac \eps 2  \rho_0 ^2 \quad \mbox{in $B_{\rho_0}$}, \quad \quad |{\bf b}_1'-{\bf b}_1|\le C_0 \eps.$$

Moreover, a similar analysis as in Proposition \ref{PM} can be performed. If ${\bf b}_1/\delta \eps^{1/2}$ is at distance at most $\mu_0$ (with $\mu_0$ small universal) away from the line $\{s \tau| s \in \R\}$ then, as in the last part of the proof of Proposition \ref{PM}, after a rotation of coordinates as in \eqref{change} we may reduce to the case when ${\bf b}_1$ satisfies the improved bound $|{\bf b}_1| \le \delta \eps^{1/2} /4$. Then  ${\bf u} \in \mathcal S (\rho_0,{\bf p}_0, \frac \eps 2)$ and the approximate solutions 
${\bf v}_1$, ${\bf v}_{\rho_0}$ for ${\bf u}$ in $B_1$ respectively $B_{\rho_0}$ satisfy
$|{\bf v}_1-{\bf v}_{\rho_0}| \le C \eps$. 

Assume now that ${\bf b}_1/\delta \eps^{1/2}$ is at distance greater than $\mu_0/2$ away from the line $\{s \tau| s \in \R\}$. Then in the proof of Proposition \ref{P5}, by Corollary 
\ref{cor31}, the right hand side of $\triangle {\bf w}$ is constant in each quadrant in $\{ x_2>0\}$ but has a discontinuity jump greater than 
$c(\delta, \mu_0)>0$ across $\{x_1=0\}$. This implies that ${\bf w}$ cannot be homogenous of degree 2 in the annulus $B_{1/2} \setminus B_{1/4}$ which, as in Proposition \ref{PM} implies the energy inequality 
\begin{equation}\label{wineq}
W({\bf u}, \rho_0) \le W({\bf u},1) - c \eps^2,
\end{equation}
for some $c$ small depending on $\delta$ and $\mu_0$.

b) If $|{\bf d}_0| \ge c_1$ then we satisfy the conclusion 
 $$ |{\bf u} - {\bf p}(\cdot, {\bf b}_0',{\bf b}_1')| \le \eps   \rho_1 ^2 \quad \mbox{in $B_{\rho_1}$, and} \quad \quad |{\bf b}_0'| \ge c_1 \eps,$$
for some small $\rho_1$, provided that $\delta$ is chosen small, depending on $\rho_1$.
 
\end{rem}

Next we show that when we end up in the situation b), then the $N$-membrane problem near the origin can be reduced to one involving fewer membranes. For this we need to iterate Proposition \ref{P5} from scale 1 to scale $\eps^{1/2}$. Precisely, let us assume that, as a starting point we have
$$ |{\bf u} - {\bf p}(\cdot, {\bf b}_0,{\bf b}_1)| \le \eps \rho_1^2 \quad \mbox{in $B_{\rho_1}$},$$ 
with $$|{\bf b}_0| \le \frac \eps 2, \quad \quad |{\bf b}_1| \le \delta \eps^{1/2}.$$ 
We can iterate the Proposition with $r=\rho_1^m$ till $r \sim \eps^{1/2}$ and obtain
\begin{equation}\label{ree12}
|{\bf u} - {\bf p}(\cdot, \bar {\bf b}_0, \bar {\bf b}_1)| \le \eps r^2 \quad \mbox{in $B_r$}, \quad \mbox{with $r=\eps^{1/2}$,}
\end{equation}
with 
\begin{equation}\label{barbf}
|\bar {\bf b}_0-  {\bf b}_0 | \le 2 C_0 \rho_1 \eps, \quad \quad |\bar {\bf b}_1-  {\bf b}_1 | \le  C |\log \eps | \eps, 
\end{equation}
(in the last step of the iteration we applied the proposition for some $\rho \in [\rho_1, \rho_1^2]$.)
Here $\rho_1$ is chosen small such that $4C_0 \rho_1 \le  c_1 \le 1$ (see \eqref{csub1}) and throughout the iteration the inequalities
$$|\bar {\bf b}_0| \le \eps, \quad |\bar {\bf b}_1| \le 2 \delta \eps^{1/2},$$
are satisfied. Moreover, if $|{\bf b}_0| \ge c_1 \eps$ then $|\bar {\bf b}_0| \ge \frac{c_1}{2} \eps$.

We rescale \eqref{ree12} to the unit ball and obtain that
$$ | r^{-2} {\bf u}(rx) - {\bf p}(x, r^{-1}\bar {\bf b}_0, \bar {\bf b}_1)| \le \eps \quad \mbox{if $x \in B_1$,} \quad \quad r=\eps^{1/2}.$$
If $0$ belongs to one of the free boundaries of ${\bf u}$, say $0 \in \Gamma_{i_0}$, and $|{\bf b}_0| \ge c_1 \eps$ then we are in the setting of Remark \ref{RFB}. Precisely we find that in $B_1$, $r^{-1}\Gamma_{i_0}$ is the free boundary of a solution $\tilde u_r$ to a multiple membrane problem involving fewer membranes, which satisfies back the hypothesis \eqref{6000} with the same value of $\eps$. We summarize the above discussion in the next lemma.

\begin{lem}\label{dic3}
Assume that ${\bf u} \in \mathcal S(1,{\bf p}_0, \eps)$ for some $\eps \le \eps_0$, i.e.
$$ |{\bf u} - {\bf p}(\cdot, 0,{\bf b}_1)| \le \eps \quad \mbox{in $B_1$, with} \quad |{\bf b}_1| \le \delta \eps^{1/2},$$
and $0 \in \Gamma_{i_0}({\bf u})$, for some $i_0$. Then one of the following alternative hold

a) $$ |{\bf u} - {\bf p}(\cdot, 0,{\bf b}'_1)| \le \frac \eps 2 \rho_0^2 \quad \mbox{in $B_{\rho_0}$, and} \quad |{\bf b}'_1-{\bf b}_1| \le C_0 \eps,$$

b) $$\Gamma_{i_0} \cap B_r \subset \{ |x_n| \le C \eps^{1/2} r\} \quad \mbox{ if $r \in [\eps^{1/2},1]$.}$$ 
When $r= \eps^{1/2}$, $\Gamma_{i_0}$ is a free boundary
 to a solution $\tilde {\bf u}$ to the multiple membrane problem in $B_r$ involving fewer membranes than $N$. Moreover, $\tilde {\bf u}$ satisfies
 $$ |\tilde {\bf u} - \tilde {\bf p}(\cdot, \tilde {\bf b}_0, \tilde {\bf b}_1)| \le 2\eps r^2 \quad \mbox{in $B_r$}, \quad \quad |\tilde {\bf b}_0| \le (2\eps)^{1/2}r, \quad |{\bf b}_1| \le \delta (2 \eps)^{1/2}.$$
Also $0 \notin \cap \Gamma_i$.

\end{lem}

Alternative b) reduces the situation to one involving fewer membranes.

It remains to investigate alternative a). While ${\bf u}$ improves at a $C^{2,\alpha}$ rate as we zoom in $B_{\rho_0}$, the bound on the size of ${\bf b}_1$ can deteriorate. Part a) implies that 
\begin{equation}\label{eps'}
{\bf u} \in \mathcal S(\rho_0,{\bf p}_0, \eps') \quad \mbox{ with} \quad \eps'=\eps + C(\delta) \eps^{3/2}.
\end{equation} 

As we iterate part a) we want to show that the approximating polynomials converge. It suffices to prove the following lemma.

\begin{lem}\label{Preg}
Assume that the hypothesis of Lemma \ref{dic3} hold and ${\bf u}$ satisfies the alternative a). Then either a1) or a2) below hold

a1)
\begin{equation}\label{a1}
{\bf u} \in \mathcal S(\rho_0,{\bf p}_0, \frac \eps 2),
\end{equation}

a2)
\begin{equation}\label{a2}
{\bf u} \in \mathcal S(\rho_0,{\bf p}_0, 2\eps) \quad \mbox{and} \quad  W({\bf p}_0) + c \eps^{3/2} \le W({\bf u},\rho_0) \le W({\bf u},1) - c \eps^2.
 \end{equation}
 
 In both cases $|{\bf v}_1-{\bf v}_{\rho_0}|_{L^\infty(B_1)} \le C \eps$ where ${\bf v}_1$, ${\bf v}_{\rho_0}$ denote the approximate solutions 
 for ${\bf u}$ in $B_1$ respectively $B_{\rho_0}$. 
\end{lem}

The Lemma is essentially included in Proposition \ref{PM} except the crucial lower bound on $W({\bf u},\rho_0)$. The statement that  $W({\bf p}_0) \le W({\bf u},\rho_0)$ allows one to 
prove the convergence of $\sum \eps_k$ as in Section 4. The inequality follows easily when $0 \in \cap \Gamma_i$ by the Weiss monotonicity formula 
and the fact that ${\bf p}_0$ is the least energy solution. However for the general case we need to establish a lower bound on the energy of 
approximate solutions the type $W({\bf p}(\cdot, {\bf b})) \ge W({\bf p}_0) - C \eps^2$.

First we establish the opposite inequality in \eqref{4100} of Lemma \ref{L41}.
\begin{lem}\label{-3}
Assume that ${\bf u} \in \mathcal S(1,{\bf p}_0, \eps)$ is $\eps$-approximated in $B_1$ by ${\bf v}:={\bf p}(\cdot,{\bf b})$. Then
$$W({\bf u},r) \ge W({\bf v}) - C(r) \eps^2.$$
\end{lem}

\begin{proof}
The proof is essentially the same as \eqref{4100} in Lemma \ref{L41} after reverting the roles of ${\bf u}$ and ${\bf v}$. We write ${\bf u}={\bf v} + \eps {\bf w},$ with $|{\bf w}| \le 1.$ 
Then, we write
$$W({\bf u}, r) = W({\bf v},r) + \eps^2 r^{n-2} I_1 +  \eps r^{n-2}  I_2,$$
with
$$I_1:= \int_{B_r}\sum \,  \frac{\omega_k}{2}|\nabla w_k|^2 dx - r^{-1} \int_{\partial B_r} \sum \omega_k w_k^2 \, d \sigma,$$
\begin{align*}
I_2:=&\int_{B_{r}} \sum \omega_k (\nabla v_k \cdot \nabla w_k + f_k w_k) dx - 
\int_{\partial B_r} \sum\omega_k  \, \frac 2 r v_k w_k d \sigma\\
= &\int_{B_{r}} \sum \omega_k (f_k -\triangle v_k) w_k dx 
\end{align*}
Now we use the fact that ${\bf v}$ is a solution in the $x_2$ variable and find (see \eqref{EL1})
$$\omega_k (f_k -\partial_{x_2x_2} v_k) w_k \ge 0.$$
Since $|\partial_{x_1x_1} v_k | \le  \delta \eps$, we find
$$\omega_k (f_k -\triangle v_k) w_k \ge \omega_k (f_k -\partial_{x_2x_2} v_k) w_k - C |\partial_{x_1x_1} v_k | \ge - C \eps,$$
which together with $I_1 \ge -C$ gives the desired conclusion.

\end{proof}

In the next lemma we show that each ${\bf p}(\cdot,{\bf b})$ $\eps$-approximates at leat one solution for which all the free boundaries intersect at the origin.

\begin{lem}\label{-2}
Given ${\bf b} \in {\bf B}({\bf p}_0)$ with $|{\bf b}| \le \delta^{1/2} \eps$, there exists ${\bf u}_{\bf b} \in \mathcal S(1,{\bf p}_0, \eps)$ with $0 \in \cap \Gamma_i$ which is $\eps$-approximated in $B_1$ by ${\bf p}(\cdot,{\bf b})$.
\end{lem}

\begin{proof}
For each solution ${\bf u}$ we associate the vector ${\bf z}\in \R^{n-1}$ given by
$$ z_i:= dist(0,\Gamma_i) \, \chi_{\{u_i=u_{i+1}\}}(0) - \sqrt{(u_i-u_{i+1})(0)} \, \chi_{\{u_i>u_{i+1}\}}(0).$$
The quadratic growth of $u_i-u_{i+1}$ away from its zero set implies that ${\bf u} \mapsto {\bf z} ({\bf u})$ is a continuous map, 
and $0 \in \Gamma_i({\bf u})$ if and only if $z_i=0$. 
Moreover, if we consider the solutions ${\bf h}(x_2,{\bf b_0})$ with free boundaries $x_2=\Gamma_i({\bf b_0})$, 
then the corresponding $z_i$ satisfies 
\begin{equation}\label{z/g}
c \le z_i / \Gamma_i({\bf b_0}) \le C.
\end{equation}

For any vector $\Gamma \in \R^{n-1}$ with $|\Gamma| \le c'$ we associate the corresponding vector ${\bf b}_0(\Gamma)  \in {\bf B}({\bf p}_0)$ for which $h(x_2,{\bf b}_0)$ has free boundaries $\Gamma$. Recall from Section 2 that $\Gamma \mapsto {\bf b}_0(\Gamma)$ is a bi-Lipschitz map. We choose $c'$ small universal such that $|{\bf b}_0| \le 1/2$.

We consider the solutions ${\bf u}_{\Gamma}$ in $B_1$ with boundary data 
${\bf p}(x,\eps {\bf b}_0(\Gamma), {\bf b})$. 
We claim that one of these functions satisfies the conditions of the Lemma.

Notice that since ${\bf p}(x,\eps {\bf b}_0, {\bf b})$ solves the Euler-Lagrange equations with error $\delta \eps$ we know that
$$ |{\bf u}_{\Gamma}-{\bf p}(x,\eps {\bf b}_0, {\bf b})| \le \delta \eps \quad \mbox{in $B_1$.}$$
On the other hand, by Lemma \ref{hxb},
$${\bf p}(x,\eps {\bf b}_0, {\bf b})={\bf p}(x,{\bf b}) + \eps x {\bf b}_0 + O(\eps^{3/2}),$$
which imply that ${\bf u}_\Gamma$ is $\eps$-approximated in $B_1$ by ${\bf p}(\cdot,{\bf b})$.

If $\delta$ is sufficiently small then
$$ |{\bf u}_{\Gamma}-{\bf p}(x,\eps {\bf b}_0, {\bf b})| \le \eps \rho^2_1 \quad \mbox{in $B_{\rho_1}$.}$$
and the arguments before Lemma \ref{dic3} applies. In particular the free boundaries of the rescaling $$\tilde {\bf u}_{\Gamma}(x):=r^{-2}{\bf u}_{\Gamma}(rx) \quad \mbox{ with $r= \eps^{1/2}$,}$$ 
are in $B_{1/2}$ in a $C \delta \eps^{1/2}$ neighborhood of the free boundaries of ${\bf h}(x_2,r \bar {\bf b}_0)$ for some $\bar {\bf b}_0$ that satisfies $$|\bar {\bf b}_0- {\bf b}_0| \le 2C_0 \rho_1,$$
(see Remark \ref{RFB} and \eqref{ree12}-\eqref{barbf} with ${\bf b}_0$, $\bar{\bf b}_0$ replaced by $\eps {\bf b}_0$ and $\eps \bar{\bf b}_0$).

 Thus the free boundaries of $\tilde {\bf u}_\Gamma$ are in a $c(\rho_1,\delta) \eps^{1/2}$ neighborhood of the free boundaries of ${\bf h}(x_2, \eps^{1/2} {\bf b}_0)$ with $c(\rho_1,\delta)  \to 0$ as $\rho_1$, $\delta \to 0$.

This means that the vector $${\bf y}_{\Gamma}:=\eps^{-1/2}{\bf z}(\tilde {\bf u}_\Gamma)$$ associated to the rescaled solution $\tilde {\bf u}_\Gamma$ above is in a $c(\rho_1,\delta)$ neighborhood of the vector $${\bf z}_\Gamma := {\bf z}({\bf h}(x_2, {\bf b}_0)).$$ corresponding to ${\bf h}(x_2, {\bf b}_0)$.

We can find the desired solution to ${\bf y}_\Gamma=0$ by a standard topological argument. 
Indeed, by \eqref{z/g} we know that $\Gamma \cdot {\bf z}_\Gamma \sim |\Gamma|^2$ hence $ \Gamma \cdot {\bf z}_\Gamma \ge c_1>0$ when $|\Gamma|=c'$. 
Then  $\Gamma \cdot {\bf y}_\Gamma >0$ when $\Gamma \in \partial B_{c'}$ provided that $c_1(\rho_1,\delta)$ is sufficiently small. 
This implies that we can find $\Gamma \in B_{c'}$ such that ${\bf y}_\Gamma=0$.

\end{proof}

As a corollary of Lemma \ref{-2} we obtain by \eqref{4100} that if $|{\bf b}| \le \delta \eps^{1/2}$ then
\begin{equation}\label{wpb}
W({\bf p}(\cdot, {\bf b})) \ge W({\bf u}_{\bf b},1/2) - C \eps^2 \ge W({\bf p}_0) - C \eps^2,
\end{equation}
where ${\bf u}_{\bf b}$ is the solution provided by Lemma \ref{-2}. 

The lower bound on $W({\bf p}(\cdot, {\bf b})) $ can be improved when ${\bf b}/\delta \eps^{1/2}$ is at distance greater than $\mu_0$ away from the 
line $\{s \tau| s \in \R\}$. For this we apply inductively Proposition \ref{P5} from scale 1 to scale $r=\eps^{1/2}$ to the function ${\bf u}_{\bf b}$ of Lemma \ref{-2}. Notice that we cannot end up in alternative b) of Remark \ref{dyc} (or Lemma \ref{dic3}) since $0 \in \cap \Gamma_i$. The iteration requires $m_0 \sim |\log \eps|$ steps and the distance from the corresponding sequence of ${\bf b}_1$'s to the $s \tau$-line remains greater than $\mu_0/2$ throughout. From Remark \ref{dyc} part a) we obtain that (see \eqref{wineq})
$$ W({\bf p}_0) \le W({\bf u}_{\bf b},\rho_0^m) \le W({\bf u}_{\bf b}, \rho_0) - (m-1) c \eps^2,$$
hence
$$ W({\bf u}_{\bf b},\rho_0) \ge W({\bf p}_0)  + c |\log \eps| \eps^2.$$
Then, by the first inequality in \eqref{wpb} we find
\begin{equation}\label{wpb2}
W({\bf p}(\cdot, {\bf b})) \ge  W({\bf p}_0)  + c |\log \eps| \eps^2.
\end{equation}
In the next lemma we show that the right hand side can be improved further, and obtain the reversed inequality to \eqref{4101} in Lemma \ref{L41}.

\begin{lem}\label{wpb3}
$$W({\bf p}(\cdot, {\bf b})) \ge  W({\bf p}_0)  + c \eps^{3/2},$$
if ${\bf b}/\delta \eps^{1/2}$ is at distance greater than $\mu_0$ away from the 
line $\{s \tau| s \in \R\}$.
\end{lem}

\begin{proof}
We claim that if ${\bf v}:= {\bf p}(\cdot, {\bf b})$, with ${\bf b}=\eps^{1/2} {\bf d}$ for some ${\bf d}$ with $|{\bf d}| \le 1$ then
\begin{equation}\label{cl6}
W({\bf v}) = \eps^{3/2} g({\bf d}) + O (\eps^2),
\end{equation}
for some continuous function $g({\bf d})$. The inequality \eqref{wpb2} implies that if ${\bf d}$ is at distance greater than $\delta \mu_0$ away from the 
line $\{s \tau| s \in \R\}$, then $g({\bf d}) >0$ and the lemma easily follows. It remains to prove the claim \eqref{cl6}.

Since ${\bf v}$ is homogenous of degree 2 we find
$$ W({\bf v}) = \int_{B_1} (v_i f_i - \frac 12 v_i \triangle v_i)\omega_i \, dx.$$
Using the same formula for ${\bf p}_0$ and that $$\int_{B_1} (v_i \triangle p_{0,i}  - p_{0,i }\triangle v_i) \omega_i \, dx  =0,$$
we get
$$W({\bf v})- W({\bf p}_0) = \int_{B_1} (v_i - p_{0,i})(f_i - \frac 12 \triangle v_i-  \frac 12 \triangle p_{0,i})\omega_i \, dx. $$
We split the integral on the right hand side into 3 angular regions: $$A_1:=\{|x_2| \le C \eps^{1/2} |x_1| \}, \quad A_2:=\{x_2 > C \eps^{1/2} |x_1| \}, \quad A_3:=\{x_2 < C \eps^{1/2} |x_1| \}.$$
In $A_3$, ${\bf v}={\bf p}_0 = 0$ and the integral is $0$. We show that the integrals in $A_1 \cap B_1$ and $A_2 \cap B_1$ have the same form as the right hand side of \eqref{cl6}.

In $A_2\cap B_1$, this follows easily from Lemma \ref{L31} which gives
$$ v_i - p_{0,i}= \eps^{1/2} d_i x_1 x_2 + O(\eps),$$
$$ f_i - \frac 12 \triangle v_i-  \frac 12 \triangle p_{0,i}=-  \eps  (e_i({\bf d}) \chi_{\{x_1>0\}} + e_i(-{\bf d}) \chi_{\{x_1<0\}}).$$
In $A_1\cap B_1$ we use that $|{\bf v}|, |{\bf p}_0| \le C \eps$, and we replace the integral in $A_1 \cap B_1$ by the integral in $T_\eps:=A_1 \cap \{ |x_1| <1\}$ since their difference is $O (\eps ^{5/2})$. Also we may replace our function by
$$w_\eps:=(v_i - p_{0,i})(f_i - \frac 12 \partial_{22} v_i-  \frac 12 \triangle p_{0,i})\omega_i $$
which differs from the original function by $O(\eps^2)$, and we integrate them in a domain of measure $\sim \eps^{1/2}$. However, the function $w_\eps$ is obtained from $w_1$ by the quadratic rescaling in the second variable $w_{\eps}(x_1,x_2)= \eps w_1(x_1, x_2/\eps^{1/2})$ which means that
$$ \int_{T_\eps} w_{\eps} dx = \eps^{3/2} \int_{T_1} w_{1} dx.$$
The claim follows since the right hand side depends (continuously) only on ${\bf d}$.

\end{proof}

\

{\it Proof of Lemma \ref{Preg}.}
We distinguish two cases as in Remark \ref{dyc} part a) depending on whether or not ${\bf b}_1/\delta \eps^{1/2}$, with ${\bf b}_1$ as in Lemma \ref{dic3}, is $\mu_0$ close to the $s \tau$-line. If ${\bf b}_1/\delta \eps^{1/2}$ is $\mu_0$ close to this line then we already showed in Remark \ref{dyc} that alternative \eqref{a1} holds. Otherwise the alternative \eqref{a2} holds since, by Lemma \ref{-3} and Lemma \ref{wpb3}
$$W({\bf u}, \rho_0) \ge W({\bf p}(\cdot, {\bf b})) - C \eps^2 \ge W({\bf p}_0) + c \eps^{3/2}.$$

\qed

\

The proof of Theorem \ref{Treg} follows from the following lemma.

\begin{lem}\label{-1}
Assume that $0 \in \Gamma_{i_0}$ and for some $\eps \le \eps_0$ small, and with $r=1$,
\begin{equation}\label{indh}
|{\bf u} - {\bf p}(\cdot, {\bf b}_0,{\bf b}_1)| \le \eps r^2\quad \mbox{in $B_r$},\mbox{for some} \quad |{\bf b}_0| \le \eps^{1/2} r, \quad |{\bf b}_1| \le \delta \eps^{1/2}.
\end{equation}
Then there exists a unit direction $\nu$ with $|\nu-e_2| \le C \eps^{1/2}$ such that
$$ \Gamma_{i_0} \subset \left \{|x\cdot \nu| \le C|x| \left( \eps^{-1/2} + |\log|x||\right)^{-1} \right \}.$$
\end{lem}

\begin{proof}
We prove the statement by induction depending on the number $N$ of membranes.

We iterate Proposition \ref{P5} in $B_{\rho_0^m}$ as long as the hypotheses are satisfied. We want to show that
$$\Gamma_{i_0} \cap B_{\rho_0^m} \subset  \left \{|x\cdot \nu| \le C \rho_0^m (\eps^{-1/2} +m)^{-1}\right \}.$$
We distinguish several cases.

{\it Case 1:} $|{\bf b}_0| \ge 3 C_0 \eps$.

We apply Proposition \ref{P5} by keeping $\eps$ fixed through the iteration (by replacing $\eps/2$ by $\eps$ in \eqref{6002}). Denote by 
${\bf b}_0^m$, ${\bf b}_1^m$ the corresponding vectors in $B_{\rho_0^m}$, and we stop the iteration when ${\bf b}^m_0 > \eps^{1/2} \rho_0^m$. By \eqref{6003}, throughout the iteration $|{\bf b}_0-{\bf b}^m_0| \le 2C _0 \eps$ (provided that $\rho_0$ is chosen small) hence the iteration stops when 
$r_m=\rho^m_0 \sim |{\bf b}_0| \eps^{-1/2} \ge \eps^{1/2}.$ Then we end up in the situation of alternative b) in Lemma \ref{dic3} with $r=r_m$. We may apply the induction hypothesis in $B_r$ (with $\eps$ replaced by $2 \eps$) to the problem involving fewer membranes, and reach the desired result.

\

{\it Case 2:} $|{\bf b}_0| \le 3 C_0 \eps$.

We may replace ${\bf b}_0$ by $0$ and $\eps$ into $C \eps$. After relabeling $\eps$ we reduce to the situation ${\bf u} \in \mathcal S(1,{\bf p}, \eps)$ of Lemma \ref{dic3}.

We iterate Lemmas \ref{dic3} and \ref{Preg} accordingly in $B_{\rho_0^m}$. 

We discuss the estimates as long as we remain in alternative a). By Lemma \ref{Preg}, we obtain that ${\bf u} \in \mathcal S(\rho_0^m,{\bf p}, \eps_m)$ for a sequence $\eps_m$, and the approximating solutions 
${\bf v}_m:= {\bf p}(\cdot, {\bf b}_1^m)$ satisfy
$\|{\bf v}_m - {\bf v}_{m+1}\|_{L^\infty(B_1)} \le C \eps_m$.

Moreover, up to the last value of $m$, $m=m_0$ (possibly infinite) for which alternative a2) applies, we know that $$w_m:= W({\bf u}, \rho_0^m) - W({\bf p}_0) \ge c \eps_m^{3/2}, \quad \forall m\le m_0,$$
hence since, by Lemma \ref{L41}, $w_m \le C \eps_m^{3/2}$ we find that for some $c_1$, $c_1'$ small
$$ a_{m+1} \le a_{m} - c_1 \eps_{m}^2 \le a_{m}-c_1 ' a_{m}^{4/3}, \quad \quad a_m:=w_m +2 c_1 \eps_m^2 \ge 0.$$
 This implies that $a_{m+1}^{-1/3} \ge a_m^{-1/3} + c$, hence 
 $$a_m \le (a_0^{-1/3} + c(m-k))^{-3}, \quad \quad m \le m_0.$$ Using that $a_0 \sim \eps^{3/2}$, $a_m \sim \eps_m^{3/2}$
 we find
 $$ \eps_m \le C (\eps^{-1/2} +  m)^{-2}.$$
  This inequality remains valid if we replace $m_0$ by $m_1 \ge m_0$ with $m_1$ denoting the first value of $m$ (possibly infinite) for which alternative b) holds, since by a1), the values of $\eps_m$ decay geometrically when $m$ goes from $m_0$ to $m_1$. We find
   $$\sum_k^{m_1} \eps_m \le C (\eps^{-1/2} +  k)^{-1}.$$
   This implies that 
  $$\|{\bf v}_m - {\bf v}_{k}\|_{L^\infty(B_1)} \le C (\eps^{-1/2} +  k)^{-1}, \quad \mbox{if $k \le m \le m_1.$}$$
  Then the angle between the rotation directions $\nu_m$ and $\nu_k$ of ${\bf v}_m$, ${\bf v}_{k} $ satisfy the same inequality, and we can use the inequality $\eps_{k}^{1/2} \le C (\eps^{-1/2} +  k)^{-1}$ to deduce that
  $$ \Gamma_{i_0} \cap B_r  \subset \left \{|x\cdot \nu_{m_1}| \le Cr (\eps^{-1/2} +  k)^{-1} \right \} \quad \mbox{if $r \ge \rho_0^k$, $k \le m_1$}.$$
   
 By Lemma \ref{dic3}, part b), the inclusion holds also when $\rho_0^{m_1} \ge r \ge \eps_{m_1}^{1/2} \rho_0^{m_1}$  with $k$ replaced by $m_1$. In the ball of radius $\eps_{m_1}^{1/2} \rho_0^{m_1}$ we can apply the induction hypothesis to obtain that
 $$ \Gamma_{i_0} \cap B_r  \subset \left \{|x\cdot \bar \nu| \le Cr (\eps_{m_1}^{-1/2} +  m-m_1)^{-1} \right \} \quad \mbox{if $r = \rho_0^m \le \eps_{m_1}^{1/2} \rho_0^{m_1}$},$$
 for some direction $\bar \nu$ with $|\bar \nu - \nu_{m_1}| \le C \eps_{m_1}^{1/2}.$ We obtain the desired conclusion with unit direction given by $\bar \nu$ since $\eps_{m_1}^{1/2} \le C (\eps^{-1/2} +  m_1)^{-1}$.

 \end{proof}

\end{document}